\documentclass[12pt,reqno]{amsart}
\usepackage{amsmath,amssymb,amsthm,mathtools}
\mathtoolsset{showonlyrefs=true}
\usepackage{mathrsfs}
\usepackage{bm}
\usepackage{enumerate}
\usepackage[authoryear,round]{natbib}
\numberwithin{equation}{section}
\allowdisplaybreaks[2]
\usepackage{xcolor}
\usepackage{amscd}
\usepackage{aliascnt}
\theoremstyle{plain}
\newtheorem{theorem}{Theorem}[section]
\newtheorem*{maintheorem}{Main theorem}
\newaliascnt{proposition}{theorem}
\newtheorem{proposition}[proposition]{Proposition}
\aliascntresetthe{proposition}
\newaliascnt{lemma}{theorem}
\newtheorem{lemma}[lemma]{Lemma}
\aliascntresetthe{lemma}
\newaliascnt{corollary}{theorem}
\newtheorem{corollary}[corollary]{Corollary}
\aliascntresetthe{corollary}
\theoremstyle{definition}
\newaliascnt{definition}{theorem}
\newtheorem{definition}[definition]{Definition}
\aliascntresetthe{definition}
\newaliascnt{remark}{theorem}
\newtheorem{remark}[remark]{Remark}
\aliascntresetthe{remark}
\usepackage[
  unicode=true,
  colorlinks=true,
  linkcolor=blue,
  citecolor=blue,
  urlcolor=blue
]{hyperref}
\usepackage[nameinlink,noabbrev]{cleveref}
\crefname{theorem}{Theorem}{Theorems}
\Crefname{theorem}{Theorem}{Theorems}
\crefname{proposition}{Proposition}{Propositions}
\Crefname{proposition}{Proposition}{Propositions}
\crefname{lemma}{Lemma}{Lemmas}
\Crefname{lemma}{Lemma}{Lemmas}
\crefname{corollary}{Corollary}{Corollaries}
\Crefname{corollary}{Corollary}{Corollaries}
\crefname{definition}{Definition}{Definitions}
\Crefname{definition}{Definition}{Definitions}
\crefname{remark}{Remark}{Remarks}
\Crefname{remark}{Remark}{Remarks}
\crefname{section}{\S}{\S\S}
\Crefname{section}{\S}{\S\S}
\crefname{subsection}{\S}{\S\S}
\Crefname{subsection}{\S}{\S\S}
\crefname{subsubsection}{\S}{\S\S}
\Crefname{subsubsection}{\S}{\S\S}
\newcommand{\C}{\mathbb{C}}
\newcommand{\N}{\mathbb{N}}
\newcommand{\Q}{\mathbb{Q}}
\newcommand{\R}{\mathbb{R}}
\newcommand{\Z}{\mathbb{Z}}
\DeclareMathOperator{\Ann}{Ann}
\DeclareMathOperator{\Coeff}{Coeff}
\DeclareMathOperator{\Hom}{Hom}
\DeclareMathOperator{\Ker}{Ker}
\DeclareMathOperator{\Mat}{Mat}
\DeclareMathOperator{\NS}{NS}
\DeclareMathOperator{\Rec}{Rec}
\DeclareMathOperator{\Rexp}{Rexp}
\DeclareMathOperator{\Ser}{Ser}
\DeclareMathOperator{\Sol}{Sol}
\DeclareMathOperator{\Supp}{Supp}
\DeclareMathOperator{\lexp}{lexp}
\DeclareMathOperator{\length}{length}
\DeclareMathOperator{\nsupp}{nsupp}
\DeclareMathOperator{\rank}{rank}
\DeclareMathOperator{\supp}{supp}
\title[Coefficient Modules over Euler Fibers]{Finite Directional Coefficient Modules over Euler Fibers of $A$-Hypergeometric Systems}
\author[Go Okuyama]{Go Okuyama}
\address{Higher Education Support Center, Hokkaido University of Science, Sapporo 006-8585, Japan}
\email{gokuyama@hus.ac.jp}
\urladdr{https://orcid.org/0000-0003-4930-7911}
\begin{document}
\begin{abstract}
Let $A\in\Mat_{d\times n}(\Z)$ have rank $d$, let $\bm{\beta}\in\C^{d}$, and consider the $A$-hypergeometric system defined by the ideal $H_{A}(\bm{\beta})$, with relation lattice $L=\Ker_{\Z}(A)$.
No homogeneity, pointedness, or positivity assumption is imposed on $A$.
We study directed formal logarithmic solutions simultaneously over all exponent lattices in the Euler fiber $A_{\C}^{-1}(\bm{\beta})$, which form an uncountable family when $L$ has positive rank.
For a fixed $L$-generic direction, negative-support strata are described by rational sign polyhedra, and their directional boundedness is characterized by recession cones.
On each exponent lattice, normalization reduces the infinitely indexed coefficient equations to a finite system of Euler, transit, and directional vanishing relations, represented by a finitely presented directional coefficient module.
We prove that only finitely many exponent lattices contribute nonzero directed solution spaces and that every contributing coefficient module has finite length.
Macaulay inverse-system duality then identifies the dimension of the all-lattice directed formal logarithmic solution space with the length of the all-lattice coefficient module.
We also give an explicit formal-series realization and a polynomial colon-ideal criterion for all realizable lowest exponents.
\end{abstract}
\maketitle
\section{Introduction}
\label{sec:introduction}
The theory of $A$-hypergeometric systems associates systems of partial differential equations with integer configurations and complex parameters.
It originated in the work of Gel'fand, Kapranov, and Zelevinsky and subsequently developed through interactions among toric geometry, combinatorial commutative algebra, Gr\"{o}bner deformation, and the algebraic analysis of hypergeometric functions \citep{GKZ89,SST,Adolphson}.
Multivariate Frobenius-type series can be constructed through Gr\"{o}bner degenerations and indicial ideals by the canonical-series algorithm in \cite[Chapter~2]{SST}.
Logarithm-free canonical series and their dimensions were studied by \citet{SaitoLogFree}, while formal Nilsson solutions for irregular $A$-hypergeometric systems with respect to generic positive weights were developed in \cite{DMM}.
These results show that weights and exponent data play a central role in the construction and organization of series solutions.

We recall the system in the generality used throughout this paper.
Let
\[
A=(\bm{a}_{1},\dots,\bm{a}_{n})=(a_{ij})\in\Mat_{d\times n}(\Z)
\]
have rank $d$.
We impose no assumption that the columns of $A$ lie in an affine hyperplane, and we impose no pointedness or positivity condition on the cone generated by them.
Let $\bm{x}=(x_{1},\dots,x_{n})^{T}$, let $\partial_{\bm{x}}=(\partial_{x_{1}},\dots,\partial_{x_{n}})^{T}$, and let
\[
D:=\C\langle\bm{x},\partial_{\bm{x}}\rangle:=\C\langle x_{1},\dots,x_{n},\partial_{x_{1}},\dots,\partial_{x_{n}}\rangle
\]
be the $n$-th Weyl algebra.
For $\bm{u}\in\Z^{n}$, write $\bm{u}=\bm{u}_{+}-\bm{u}_{-}$ with disjointly supported vectors $\bm{u}_{+},\bm{u}_{-}\in\N^{n}$, and put
\[
L:=\Ker_{\Z}(A).
\]
The toric ideal is
\[
I_{A}:=\left\langle\partial_{\bm{x}}^{\bm{u}_{+}}-\partial_{\bm{x}}^{\bm{u}_{-}}\,\middle\vert\,\bm{u}\in L\right\rangle\subset\C[\partial_{\bm{x}}].
\]
Writing $\theta_{x_{j}}:=x_{j}\partial_{x_{j}}$ and $\theta_{\bm{x}}:=(\theta_{x_{1}},\dots,\theta_{x_{n}})^{T}$, define the Euler operators by
\[
(A\theta_{\bm{x}})_{i}:=\sum_{j=1}^{n}a_{ij}\theta_{x_{j}}
\qquad
(i=1,\dots,d).
\]
For $\bm{\beta}=(\beta_{1},\dots,\beta_{d})^{T}\in\C^{d}$, the $A$-hypergeometric ideal is
\begin{equation}
\label{eq:introduction-A-hypergeometric-ideal}
H_{A}(\bm{\beta}):=DI_{A}+\sum_{i=1}^{d}D\bigl((A\theta_{\bm{x}})_{i}-\beta_{i}\bigr),
\end{equation}
and the corresponding $A$-hypergeometric $D$-module is $D/H_{A}(\bm{\beta})$.
In the homogeneous setting, this system goes back to the hypergeometric system associated with a finite configuration of torus characters introduced by Gel'fand, Zelevinsky, and Kapranov \cite[Introduction and equation~(0.2)]{GKZ89}; the present notation follows the now standard Weyl-algebra formulation.
Holonomicity without the classical homogeneity restriction is established in \cite[Theorem~3.9]{Adolphson}.
The present paper does not compare the formal solution spaces constructed here with the holonomic rank, except for an independent numerical comparison in the final example.

Logarithmic series solutions have been approached through multivariate versions of Frobenius's method.
A perturbation construction for logarithmic $A$-hypergeometric series was introduced in \cite{Log1}, and fixed-exponent coefficient spaces with sufficient conditions for obtaining all solutions at a selected exponent were developed in \cite{Log2}.
The extended Frobenius framework in \cite{Log3} treats all $A$-hypergeometric series solutions and gives a dual formulation in terms of constant-coefficient differential operators.
The comparison problem between ambient and intrinsic perturbations raised in \cite{Log3} was subsequently resolved in \cite{NakanoObstruction}.
For a fixed fake exponent and an ordered negative-support family, intrinsic perturbation was shown to produce, in general, a proper subspace of the ambient coefficient space, and the discrepancy was organized into an intrinsic-perturbation obstruction module defined from the corresponding colon ideals.
The construction in the present paper is independent of this comparison: rather than selecting a perturbation mechanism at a fixed fake exponent, it represents the full classwise coefficient space directly by the polynomial dual of the finite directional coefficient module and organizes these modules over the entire Euler fiber.
Complementary computability criteria for parameters in the core were obtained in \cite{NagamineCore}, and an Artinian and Hilbert-series description of leading logarithmic coefficient spaces attached to a fixed fake exponent of a homogeneous system was given in \cite{Nagamine}.
The latter construction is attached to a fixed fake exponent and its standard-pair data, whereas the present module is defined on an entire exponent lattice and records all normalized coefficient components indexed by its negative-support strata.
The fixed-exponent perspective in these works is indispensable for explicit construction, but it leaves a prior organizational question: if no fake exponent is selected in advance, which exponent classes in the entire Euler fiber can contribute directed formal logarithmic solutions?

The present paper addresses this question.
Because $A$ has rank $d$, the complex-linear map $A_{\C}:\C^{n}\to\C^{d}$ is surjective, and the Euler fiber $A_{\C}^{-1}(\bm{\beta})$ is a nonempty affine space parallel to $\Ker_{\C}(A)$.
For $\bm{\lambda}\in A_{\C}^{-1}(\bm{\beta})$, the coset $\Lambda(\bm{\lambda}):=\bm{\lambda}+L$ is called an exponent lattice.
When $n>d$, the Euler fiber contains continuum many exponent lattices, although each individual lattice is countable.
The first global problem is to replace this uncountable family by a finite, intrinsically defined collection of contributing exponent lattices.
Unlike the finite fake-exponent sets obtained from a fixed initial ideal and its standard pairs \cite[Chapter~3]{SST}, the present finite family is extracted directly from the entire Euler fiber and does not require a Gr\"{o}bner degeneration.

Fix an $L$-generic vector $\bm{w}\in\R^{n}$, meaning that $\bm{w}\cdot\bm{u}\neq0$ for every nonzero $\bm{u}\in L$.
It defines a total order inside each exponent lattice, but no order between distinct exponent lattices.
Each exponent lattice is stratified by negative support.
Every stratum is represented by the integral points of a rational sign polyhedron in Gale coordinates, and boundedness below in the $\bm{w}$-order is equivalent to nonnegativity of the induced functional on its recession cone by \cref{prop:directional-boundedness-stratum-minima}.
This separates the directional geometry from the subsequent coefficient algebra and works without homogeneity or pointedness assumptions on $A$.

After an exponent lattice is embedded in its $\Z^{n}$-stable ambient exponent component, the Euler equations force formal logarithmic solutions back onto the Euler fiber by \cref{prop:ambient-coefficient-criterion-Euler-support-reduction}.
Within one exponent lattice, the toric equations compare raw coefficients at infinitely many lattice points.
A lattice-wide normalization reduces these equations to finitely many polynomial components indexed by the negative supports occurring in that lattice.
The normalized Euler equations, the squarefree transit equations, and the vanishing of directionally unacceptable components define the finite directional coefficient module $M_{\Lambda,\bm{w}}^{\mathrm{dir}}$.
The classwise presentation and inverse normalization yield
\begin{equation}
\label{eq:introduction-classwise-directional-isomorphism}
\Sol_{\bm{w},\Lambda}^{\mathrm{dir}}\bigl(H_{A}(\bm{\beta})\bigr)\simeq\Coeff_{\bm{w}}^{\mathrm{dir}}(\Lambda)\simeq\Hom_{S}\left(M_{\Lambda,\bm{w}}^{\mathrm{dir}},\C[\bm{y}]\right),
\end{equation}
where $S=\C[\partial_{y_{1}},\dots,\partial_{y_{n}}]$ acts on $\C[\bm{y}]$ by constant-coefficient differentiation; see \cref{thm:directed-coefficient-module-presentation}.
The normalization in \eqref{eq:introduction-classwise-directional-isomorphism} is an exponent-lattice version of the fixed-exponent coefficient comparisons in \cite[Section~4]{Log3}.
Its role here is different because the exponent lattice itself is an output of a finite global selection procedure rather than an input chosen in advance.

The first new finiteness mechanism is geometric and arithmetic.
If an exponent lattice contains a $\bm{w}$-bounded negative-support stratum, then the rows of a Gale dual matrix indexed by its integral coordinates have full column rank by \cref{lem:full-coordinate-rank-forced-by-boundedness}.
For every fixed full-rank row set, Smith normal form identifies the possible exponent lattices with a finite saturation quotient by \cref{prop:finite-classification-prescribed-integral-coordinates}.
It follows from \cref{thm:finiteness-admissible-exponent-lattice-family} that the admissible exponent-lattice family is finite.
This passage from an uncountable Euler fiber to finitely many nonzero classwise blocks is not a rank statement and does not use regular holonomicity.

The second finiteness mechanism is module-theoretic.
At a bounded inclusion-minimal negative support, the Euler linear forms and boundary monomials define a local constraint ideal whose radical is the homogeneous maximal ideal of $S$ by \cref{thm:finite-length-minimal-support-constraint-quotient}.
Each cyclic submodule generated by a bounded inclusion-minimal negative support is a quotient of the corresponding local constraint quotient.
Since these cyclic submodules generate the directional module, every nonzero classwise module has finite length and is supported at the homogeneous maximal ideal.
Consequently, the all-lattice module
\begin{equation}
\label{eq:introduction-all-lattice-module}
M_{\bm{\beta},\bm{w}}^{\mathrm{dir,all}}:= \bigoplus_{\Lambda\in\mathfrak{L}_{\bm{\beta},\bm{w}}^{\mathrm{adm}}}M_{\Lambda,\bm{w}}^{\mathrm{dir}}
\end{equation}
is a finite direct sum of such modules by \cref{thm:finiteness-admissible-exponent-lattice-family,thm:finite-length-classwise-all-lattice-modules}.
This finite-length statement is stronger than finite generation and retains nilpotent coefficient directions that need not be detected by counting possible lowest exponents.
The final example makes this distinction explicit: the all-lattice module has length ten, while only nine lowest exponents are realizable.

Finite-length Macaulay--Matlis duality converts module length into solution-space dimension.
In characteristic zero, the Matlis dual of a finite-length module supported at the homogeneous maximal ideal admits a polynomial realization under constant-coefficient differentiation; see \cite[Section~21]{Eisenbud}, \cite[Theorem~2.1]{SchulzeTozzo}, and \cite[equations~(2.1)--(2.3)]{SchulzeTozzo}.
Applying this duality to the finite all-lattice module and using the classwise isomorphisms in \eqref{eq:introduction-classwise-directional-isomorphism} gives
\begin{equation}
\label{eq:introduction-global-length-dimension}
\dim_{\C}\Sol_{\bm{\beta},\bm{w}}^{\mathrm{dir,all}}\bigl(H_{A}(\bm{\beta})\bigr)=\length_{S}\left(M_{\bm{\beta},\bm{w}}^{\mathrm{dir,all}}\right)=\sum_{\Lambda\in\mathfrak{L}_{\bm{\beta},\bm{w}}^{\mathrm{adm}}}\length_{S}\left(M_{\Lambda,\bm{w}}^{\mathrm{dir}}\right),
\end{equation}
as proved in \cref{thm:global-length-dimension-formula}.
The formal logarithmic realization in \cref{thm:formal-logarithmic-realization} reconstructs the raw coefficients and the corresponding series explicitly.
Finally, \cref{thm:finite-criterion-all-classwise-realizable-exponents,thm:polynomial-colon-ideal-criterion-classwise-realizability} determine every realizable lowest exponent by finite coordinate-kernel and polynomial colon-ideal tests.
The latter criterion is filtration-sensitive: before testing a component, it forces the vanishing of all acceptable components with earlier stratum minima.
It therefore differs from the fixed-coefficient colon ideal in \cite[Theorem~4.7]{Log3} and does not require the homogeneous indicial comparison in \cite[Theorem~4.9]{Log3}.

For convenience, the principal conclusions are collected below without assigning a theorem number.
\begin{maintheorem}
Let $A\in\Mat_{d\times n}(\Z)$ have rank $d$, let $\bm{\beta}\in\C^{d}$, put $L=\Ker_{\Z}(A)$, and fix an $L$-generic direction $\bm{w}\in\R^{n}$.
Then the following statements hold.
\begin{enumerate}[\rm (1)]
\item For every exponent lattice $\Lambda\in\mathfrak{L}_{\bm{\beta}}$, normalization identifies the classwise directed formal logarithmic solution space with the polynomial dual of the finite directional coefficient module; see \cref{thm:directed-coefficient-module-presentation}.
\item An exponent lattice contributes a nonzero directed formal logarithmic solution space if and only if it possesses a bounded inclusion-minimal negative-support stratum.
Every such stratum minimum is realizable as a lowest exponent by \cref{thm:minimal-support-solutions-nonvanishing}.
\item Only finitely many exponent lattices contribute, by \cref{thm:finiteness-admissible-exponent-lattice-family}.
\item Every contributing classwise directional coefficient module has finite length and is supported at the homogeneous maximal ideal.
The all-lattice directional coefficient module has finite length, with support contained in the corresponding singleton; see \cref{thm:finite-length-classwise-all-lattice-modules}.
\item The dimension of the all-lattice directed formal logarithmic solution space is given by \eqref{eq:introduction-global-length-dimension}.
Every normalized coefficient family is explicitly realized as a formal logarithmic series by \cref{thm:formal-logarithmic-realization}.
\item All realizable lowest exponents are determined by the finite criteria in \cref{thm:finite-criterion-all-classwise-realizable-exponents,thm:polynomial-colon-ideal-criterion-classwise-realizability}.
\end{enumerate}
\end{maintheorem}

The main theorem makes no assertion of convergence and does not identify the space in \eqref{eq:introduction-global-length-dimension} with a local holomorphic solution space or with the holonomic rank.
Such comparisons require analytic hypotheses absent from the present formal construction.
This distinction is especially important in the nonhomogeneous case, where irregular Nilsson theory involves additional convergence and asymptotic questions \citep{DMM}.
The final section records an independent agreement between the computed formal dimension and a Macaulay2 holonomic-rank calculation, but that agreement is not used in the proofs.
A different global framework for rank variation is provided by Euler--Koszul homology and local cohomology in \cite{MMW}.
The present construction instead isolates a finite coefficient module directly from exponent-lattice support geometry.

The paper is organized as follows.
Sections~\ref{sec:support-geometry} and \ref{sec:classwise-finite-presentation} develop the directional support geometry, Euler reduction, normalized coefficient system, and finite classwise module.
Sections~\ref{sec:global-finiteness} and \ref{sec:local-finite-length} prove the two finiteness mechanisms: the finiteness of the admissible exponent-lattice family and the finite length of the resulting coefficient modules.
Section~\ref{sec:dimension-series-realization} establishes the length--dimension formula, explicit formal realization, and realizability criteria, while Section~\ref{sec:complete-all-lattice-example} carries out the complete finite calculation for a rank-jumping parameter.

\section{Ambient formal series and directional support geometry}
\label{sec:support-geometry}

\subsection{Exponent lattices}
\label{subsec:exponent-lattices}

Let $A_{\C}:\C^{n}\to\C^{d}$ be the complex linear map associated with $A$.
For $\bm{\beta}\in\C^{d}$, we call the fiber $A_{\C}^{-1}(\bm{\beta})$ the \emph{Euler fiber} of $H_{A}(\bm{\beta})$.
For $\bm{\lambda}\in A_{\C}^{-1}(\bm{\beta})$, define $\Lambda(\bm{\lambda}):=\bm{\lambda}+L$ and $\widetilde{\Lambda}(\bm{\lambda}):=\bm{\lambda}+\Z^{n}$.
Since $\widetilde{\Lambda}(\bm{\lambda}) \cap A_{\C}^{-1}(\bm{\beta}) = \Lambda(\bm{\lambda})$, we have
\begin{equation}
\label{eq:ambient-components-indexed-by-L-cosets}
\widetilde{\Lambda}(\bm{\lambda})=\widetilde{\Lambda}(\bm{\lambda}') \iff \Lambda(\bm{\lambda})=\Lambda(\bm{\lambda}') \iff \bm{\lambda}-\bm{\lambda}'\in L.
\end{equation}
We call $\Lambda(\bm{\lambda})$ the \emph{exponent lattice} and $\widetilde{\Lambda}(\bm{\lambda})$ the \emph{ambient exponent component} of $\bm{\lambda}$.
Denote the set of exponent lattices in the Euler fiber $A_{\C}^{-1}(\bm{\beta})$ by $\mathfrak{L}_{\bm{\beta}}$.

The Euler fiber and its $\Z^{n}$-stable ambient enlargement admit the decompositions
\[
A_{\C}^{-1}(\bm{\beta})=\coprod_{\Lambda\in\mathfrak{L}_{\bm{\beta}}}\Lambda,\quad A_{\C}^{-1}(\bm{\beta})+\Z^{n}=\coprod_{\Lambda\in\mathfrak{L}_{\bm{\beta}}}\widetilde{\Lambda}.
\]
Here $\widetilde{\Lambda}=\Lambda+\Z^{n}$.

\subsection{Ambient logarithmic formal series and coefficient families}
\label{subsec:ambient-logarithmic-coefficient-families}
Fix an exponent lattice $\Lambda\in\mathfrak{L}_{\bm{\beta}}$ and put $\widetilde{\Lambda}:=\Lambda+\Z^{n}$.
We define a $D$-module of ambient formal logarithmic series on $\widetilde{\Lambda}$, together with its coefficient-family realization; the Euler equations will later restrict its support to $\Lambda$.
Let $\bm{y}=(y_{1},\dots,y_{n})^{T}$ be a new set of indeterminates, and let $\partial_{\bm{y}}=(\partial_{y_{1}},\dots,\partial_{y_{n}})^{T}$, where $\partial_{y_{j}}:=\frac{\partial}{\partial y_{j}}$ for $j=1,\dots,n$.

Let $S:=\C[\partial_{y_{1}},\dots,\partial_{y_{n}}]$, acting on $\C[\bm{y}]$ by the usual constant-coefficient differentiation, denoted by $\mathbin{\bullet}_{\bm{y}}$.

Define the ambient formal logarithmic series space by
\[
\mathscr{F}_{\widetilde{\Lambda}}^{\log}:=\left\{\sum_{\bm{\gamma}\in\widetilde{\Lambda}}\bm{x}^{\bm{\gamma}}p_{\bm{\gamma}}(\log\bm{x})\,\middle\vert\, p_{\bm{\gamma}}(\bm{y})\in\C[\bm{y}]\text{ for every }\bm{\gamma}\in\widetilde{\Lambda}\right\},
\]
where $\log\bm{x}:=(\log x_{1},\dots,\log x_{n})^{T}$.
The monomials $\bm{x}^{\bm{\gamma}}$ for $\bm{\gamma}\in\widetilde{\Lambda}$ are considered as formal basis symbols and equality in $\mathscr{F}_{\widetilde{\Lambda}}^{\log}$ is coefficientwise.
For $\varphi=\sum_{\bm{\gamma}\in\widetilde{\Lambda}}\bm{x}^{\bm{\gamma}}p_{\bm{\gamma}}(\log\bm{x})\in\mathscr{F}_{\widetilde{\Lambda}}^{\log}$, define its support by
\[
\Supp(\varphi):=\left\{\bm{\gamma}\in\widetilde{\Lambda}\,\middle\vert\,p_{\bm{\gamma}}(\bm{y})\neq 0\right\}.
\]

Define an action $\mathbin{\bullet}$ of $D$ on $\mathscr{F}_{\widetilde{\Lambda}}^{\log}$ as follows.
For $j=1,\dots,n$, $\bm{\gamma}\in\widetilde{\Lambda}$, and $p(\bm{y})\in\C[\bm{y}]$, define $x_{j}\mathbin{\bullet}(\bm{x}^{\bm{\gamma}}p(\log\bm{x}))$ and $\partial_{x_{j}}\mathbin{\bullet}(\bm{x}^{\bm{\gamma}}p(\log\bm{x}))$ in $\mathscr{F}_{\widetilde{\Lambda}}^{\log}$ by
\begin{align}
\label{eq:ambient-logarithmic-Weyl-action}
x_{j}\mathbin{\bullet}\bigl(\bm{x}^{\bm{\gamma}}p(\log\bm{x})\bigr)&:=\bm{x}^{\bm{\gamma}+\bm{e}_{j}}p(\log\bm{x}),\\
\label{eq:ambient-logarithmic-derivative-action}
\partial_{x_{j}}\mathbin{\bullet}\bigl(\bm{x}^{\bm{\gamma}}p(\log\bm{x})\bigr)&:=\bm{x}^{\bm{\gamma}-\bm{e}_{j}}\left((\gamma_{j}+\partial_{y_{j}})\mathbin{\bullet}_{\bm{y}}p\right)(\log\bm{x}),
\end{align}
and extend it coefficientwise to all of $\mathscr{F}_{\widetilde{\Lambda}}^{\log}$.

Define the ambient coefficient space by
\[
\mathscr{C}_{\widetilde{\Lambda}}:=\prod_{\bm{\gamma}\in\widetilde{\Lambda}}\C[\bm{y}].
\]
Consider the coefficient extraction map $\Psi_{\widetilde\Lambda}:\mathscr{F}_{\widetilde{\Lambda}}^{\log}\to\mathscr{C}_{\widetilde{\Lambda}}$ defined by
\begin{equation}
\label{eq:ambient-coefficient-extraction}
\Psi_{\widetilde\Lambda}\left(\sum_{\bm{\gamma}\in\widetilde{\Lambda}} \bm{x}^{\bm{\gamma}} p_{\bm{\gamma}}(\log\bm{x})\right):=(p_{\bm{\gamma}}(\bm{y}))_{\bm{\gamma}\in\widetilde{\Lambda}}.
\end{equation}
Its inverse is the formal logarithmic substitution map $\Phi_{\widetilde{\Lambda}}:\mathscr{C}_{\widetilde{\Lambda}}\to\mathscr{F}_{\widetilde{\Lambda}}^{\log}$ defined by
\begin{equation}
\label{eq:ambient-formal-logarithmic-substitution}
\Phi_{\widetilde{\Lambda}}\left((p_{\bm{\gamma}}(\bm{y}))_{\bm{\gamma}\in\widetilde{\Lambda}}\right):=\sum_{\bm{\gamma}\in\widetilde{\Lambda}} \bm{x}^{\bm{\gamma}}p_{\bm{\gamma}}(\log\bm{x}).
\end{equation}
Thus $\Psi_{\widetilde{\Lambda}}$ and $\Phi_{\widetilde{\Lambda}}$ are mutually inverse $\C$-linear isomorphisms.

For $\Omega\subset\widetilde{\Lambda}$, define the subspace of $\mathscr{F}_{\widetilde{\Lambda}}^{\log}$ supported on $\Omega$ by
\begin{equation}
\label{eq:Euler-fiber-supported-subspaces}
\mathscr{F}_{\widetilde{\Lambda}}^{\log}[\Omega]:=\left\{\varphi\in\mathscr{F}_{\widetilde{\Lambda}}^{\log}\,\middle\vert\,\Supp(\varphi)\subset\Omega\right\}.
\end{equation}
Its corresponding coefficient subspace of $\mathscr{C}_{\widetilde{\Lambda}}$ under $\Psi_{\widetilde{\Lambda}}$ is
\[
\mathscr{C}_{\widetilde{\Lambda}}[\Omega]:=\left\{\bm{c}\in\mathscr{C}_{\widetilde{\Lambda}}\,\middle\vert\, c_{\bm{\gamma}}=0\text{ for }\bm{\gamma}\notin\Omega\right\}.
\]

\subsection{Negative-support strata and their polyhedral realization}
\label{subsec:negative-support-stratification}
For $\bm{\gamma} = (\gamma_{1},\dots,\gamma_{n})^{T} \in \C^{n}$, its negative support is defined by
\[
\nsupp(\bm{\gamma}):=\left\{j\in [n]:=\{1,\dots,n\}\,\middle\vert\,\gamma_{j}\in\Z_{<0}\right\}.
\]
Fix an exponent lattice $\Lambda=\Lambda(\bm{\lambda})\in\mathfrak{L}_{\bm{\beta}}$.
For $I\subset[n]$, put
\[
\Lambda_{I}:=\left\{\bm{\gamma}\in\Lambda\,\middle\vert\,\nsupp(\bm{\gamma})=I\right\},
\]
and we call $\Lambda_{I}$ the \emph{negative-support stratum} of $\Lambda$ indexed by $I$.
Denote the set of negative supports occurring on $\Lambda$ by
\[
\NS(\Lambda):=\left\{I\subset [n]\,\middle\vert\,\Lambda_{I}\neq\emptyset\right\}.
\]
Then $\NS(\Lambda)$ is a finite set, and we have the following finite decomposition of $\Lambda$ into negative-support strata:
\begin{equation}
\label{eq:finite-negative-support-stratification}
\Lambda=\coprod_{I\in\NS(\Lambda)}\Lambda_{I}.
\end{equation}

For $\Lambda=\Lambda(\bm{\lambda})$, define the \emph{integral-coordinate set} by
\begin{equation}
\label{eq:integral-coordinate-set}
E_{\Lambda}:=\left\{j\in[n]\,\middle\vert\,\lambda_{j}\in\Z\right\}.
\end{equation}
This set is independent of the choice of $\bm{\lambda}\in\Lambda$ because differences of vectors in $\Lambda$ belong to $L\subset\Z^{n}$.
In particular, every $I\in\NS(\Lambda)$ is contained in $E_{\Lambda}$.

Put $r:=\rank_{\Z}(L)=n-d$ and let
\begin{equation}
\label{eq:exponent-lattice-basis-matrix}
B:=(\bm{b}_{1},\dots,\bm{b}_{r})\in\Mat_{n\times r}(\Z),
\end{equation}
be a Gale dual matrix of $A$, where $\bm{b}_{1},\dots,\bm{b}_{r}$ form a $\Z$-basis of $L$.
Every $\bm{\gamma}\in\Lambda$ has a unique expression
\[
\bm{\gamma}=\bm{\lambda}+B\bm{m}\quad (\bm{m}\in\Z^{r}).
\]

For $I\in\NS(\Lambda)$, define the sign polyhedron
\[
P_{\Lambda,I}^{\bm{\lambda}}:=\left\{\bm{m}\in\R^{r}\,\middle\vert\,\begin{aligned}
\lambda_{j}+(B\bm{m})_{j} & \leq -1 && (j\in I), \\
\lambda_{j}+(B\bm{m})_{j} & \geq 0 && (j\in E_{\Lambda}\setminus I)
\end{aligned}
\right\}.
\]
This polyhedron depends on the choice of reference vector $\bm{\lambda}\in\Lambda$, but its recession cone
\begin{equation}
\label{eq:classwise-negative-support-recession-cone}
R_{\Lambda,I}:=\Rec(P_{\Lambda,I}^{\bm{\lambda}}):=\left\{\bm{v}\in\R^{r}\,\middle\vert\,
\begin{aligned}
(B\bm{v})_{j} & \leq 0 && (j\in I),\\
(B\bm{v})_{j} & \geq 0 && (j\in E_{\Lambda}\setminus I)
\end{aligned}
\right\}
\end{equation}
is independent of $\bm{\lambda}$.

\begin{lemma}
\label{lem:integral-description-negative-support-stratum}
For every $I\in\NS(\Lambda)$, we have
\begin{equation}
\label{eq:negative-support-stratum-integral-description}
\Lambda_{I}=\left\{\bm{\lambda}+B\bm{m}\,\middle\vert\,\bm{m}\in P_{\Lambda,I}^{\bm{\lambda}}\cap\Z^{r}\right\}.
\end{equation}
\end{lemma}
\begin{proof}
Every vector of $\Lambda$ is uniquely $\bm{\lambda}+B\bm{m}$ with $\bm{m}\in\Z^{r}$.
For $j\in E_{\Lambda}$, the coordinate $\lambda_{j}+(B\bm{m})_{j}$ is integral; it is negative precisely when it is at most $-1$, and it is not negative precisely when it is at least $0$.
Coordinates outside $E_{\Lambda}$ are never integral and hence never belong to a negative support.
Thus $\nsupp(\bm{\lambda}+B\bm{m})=I$ is equivalent to the defining inequalities of $P_{\Lambda,I}^{\bm{\lambda}}$, proving the assertion.
\end{proof}

\subsection{Directional boundedness and bounded inclusion-minimal negative supports}
\label{subsec:directional-boundedness-minimal-supports}
Fix an exponent lattice $\Lambda\in\mathfrak{L}_{\bm{\beta}}$.

\begin{definition}
\label{def:L-generic-direction}
A direction vector $\bm{w}\in\R^{n}$ is called an \emph{$L$-generic direction} if $\bm{w}\cdot\bm{u}\neq 0$ for $\bm{u}\in L\setminus\{\bm{0}\}$.
\end{definition}

The set of $L$-generic directions is
\[
\mathcal{W}_{L}^{\mathrm{gen}}=\R^{n}\setminus\bigcup_{\bm{u}\in L\setminus\{\bm{0}\}}\bm{u}^{\perp}.
\]
Since $L$ is countable and each $\bm{u}^{\perp}$, with $\bm{u}\in L\setminus\{\bm{0}\}$, is a proper closed hyperplane, $\mathcal{W}_{L}^{\mathrm{gen}}$ is a dense $G_{\delta}$ subset of $\R^{n}$ by the Baire category theorem.
The present genericity condition is imposed to define a total order on every exponent lattice, not to construct a Gr\"{o}bner degeneration.
It is stronger than the usual Gr\"{o}bner-genericity ensuring that $\operatorname{in}_{\bm{w}}(I_{A})$ is monomial \citep{SST,DMM}.
If $L=\{\bm{0}\}$, every direction is $L$-generic.

Fix $\bm{w}\in\mathcal{W}_{L}^{\mathrm{gen}}$ throughout this section.
For every exponent lattice $\Lambda\in\mathfrak{L}_{\bm{\beta}}$ and every pair $\bm{\gamma},\bm{\gamma}'\in\Lambda$, define
\[
\bm{\gamma}\leq_{\bm{w},\Lambda}\bm{\gamma}'\iff \bm{w}\cdot \left(\bm{\gamma}'-\bm{\gamma}\right)\geq 0.
\]
By the $L$-genericity of $\bm{w}$, this relation is a total order on $\Lambda$, with smaller elements having smaller $\bm{w}$-weight relative to any reference point in $\Lambda$.
When the exponent lattice is clear from the context, we write $\leq_{\bm{w}}$ in place of $\leq_{\bm{w},\Lambda}$.
The same $\bm{w}$ is used as a linear functional on $L$ for every exponent lattice.
No order is defined between distinct exponent lattices, because the present theory applies $\bm{w}$ only to differences lying in $L$.

Let $\Lambda=\Lambda(\bm{\lambda})\in\mathfrak{L}_{\bm{\beta}}$ and let $\Omega\subset\Lambda$ be nonempty.
The subset $\Omega$ is called \emph{$\bm{w}$-bounded below} if $\left\{\bm{w}\cdot(\bm{\gamma}-\bm{\lambda}) \,\middle\vert\, \bm{\gamma}\in\Omega \right\}$ is bounded below in $\R$.
This property is independent of the choice of a reference vector $\bm{\lambda}\in\Lambda$.

An element $\bm{v}\in\Omega$ is called the \emph{$\bm{w}$-minimum of $\Omega$} if $\bm{v}\leq_{\bm{w},\Lambda}\bm{\gamma}$ for all $\bm{\gamma}\in\Omega$.

Let $B\in\Mat_{n\times r}(\Z)$ be as in \eqref{eq:exponent-lattice-basis-matrix}.
Then $\bm{\omega}:= B^{T}\bm{w} \in \R^{r}$ defines a linear functional on $\R^{r}$.

\begin{proposition}
\label{prop:directional-boundedness-stratum-minima}
Let $\Lambda\in\mathfrak{L}_{\bm{\beta}}$ and $I\in\NS(\Lambda)$.
Then the following conditions are equivalent.
\begin{enumerate}[\rm (1)]
\item The stratum $\Lambda_{I}$ is $\bm{w}$-bounded below.
\item The linear functional $\bm{\omega}=B^{T}\bm{w}$ is nonnegative on the recession cone $R_{\Lambda,I}$.
\item The stratum $\Lambda_{I}$ has a unique $\bm{w}$-minimum.
\end{enumerate}
\end{proposition}
\begin{proof}
Choose $\bm{\lambda}\in\Lambda$.
By \cref{lem:integral-description-negative-support-stratum}, the map $\bm{m}\mapsto\bm{\lambda}+B\bm{m}$ is a bijection from $P_{\Lambda,I}^{\bm{\lambda}}\cap\Z^{r}$ onto $\Lambda_{I}$, and
\[
\bm{w}\cdot\bigl((\bm{\lambda}+B\bm{m})-\bm{\lambda}\bigr)=\bm{w}^{T}B\bm{m}=\bm{\omega}^{T}\bm{m}=\bm{\omega}\cdot\bm{m}.
\]
Assume that \textup{(2)} does not hold.
Then, by the rationality of $R_{\Lambda,I}$, there exists a nonzero integral vector $\bm{q}\in R_{\Lambda,I}$ such that $\bm{\omega}\cdot\bm{q}<0$.
For any $\bm{m}_{0}\in P_{\Lambda,I}^{\bm{\lambda}}\cap\Z^{r}$ and $k\in\N$, we have $\bm{m}_{0}+k\bm{q}\in P_{\Lambda,I}^{\bm{\lambda}}\cap\Z^{r}$ and
\[
\bm{\omega}\cdot (\bm{m}_{0}+k\bm{q})\to-\infty\quad (k\to\infty).
\]
This shows that $\bm{\omega}$ is unbounded below on $P_{\Lambda,I}^{\bm{\lambda}}\cap \Z^{r}$ and therefore $\Lambda_{I}$ is not $\bm{w}$-bounded below, namely \textup{(1)} does not hold.
Thus \textup{(1)} implies \textup{(2)}.

Conversely, assume that \textup{(2)} holds.
Since $P_{\Lambda,I}^{\bm{\lambda}}$ is a nonempty rational polyhedron, it is the Minkowski sum of a rational polytope and its recession cone $R_{\Lambda,I}$.
The functional $\bm{\omega}$ is bounded below on the polytope and is nonnegative on the recession cone, so it is bounded below on $P_{\Lambda,I}^{\bm{\lambda}}$ and hence on $P_{\Lambda,I}^{\bm{\lambda}}\cap\Z^{r}$.
Therefore $\Lambda_{I}$ is $\bm{w}$-bounded below.

Assume these equivalent conditions \textup{(1)} and \textup{(2)}.
The functional $\bm{\omega}$ is in fact strictly positive on $R_{\Lambda,I}\setminus\{\bm{0}\}$.
Indeed, because $\bm{\omega}$ is nonnegative on the rational polyhedral cone $R_{\Lambda,I}$, the set $R_{\Lambda,I}\cap\Ker(\bm{\omega})$ is a face of $R_{\Lambda,I}$ and is therefore a rational polyhedral cone.
If this face were nonzero, it would contain a nonzero integral vector $\bm{q}$.
Then $B\bm{q}\in L\setminus\{\bm{0}\}$ and $\bm{w}\cdot B\bm{q}=\bm{\omega}\cdot\bm{q}=0$, contrary to the $L$-genericity of $\bm{w}$.
Hence $\bm{\omega}$ must be strictly positive on $R_{\Lambda,I}\setminus\{\bm{0}\}$.
Therefore, for every $c\in\R$, the polyhedron
\[
\{\bm{m}\in P_{\Lambda,I}^{\bm{\lambda}}\mid\bm{\omega}\cdot\bm{m}\leq c\}
\]
is bounded.
Taking $c=\bm{\omega}\cdot\bm{m}_{0}$ for one integral point $\bm{m}_{0}\in P_{\Lambda,I}^{\bm{\lambda}}\cap\Z^{r}$ gives a nonempty finite set of integral points on which $\bm{\omega}$ attains its minimum.
Choose a minimizing point and denote it by $\bm{m}_{\Lambda,I}$.
Then its image $\bm{\lambda}+B\bm{m}_{\Lambda,I}$ is a $\bm{w}$-minimum of $\Lambda_{I}$, and is unique by the $L$-genericity of $\bm{w}$.
This shows that \textup{(3)} holds.

Finally, it is clear that \textup{(3)} implies \textup{(1)} because a nonempty set with a minimum is necessarily bounded below.
\end{proof}
Define
\begin{align}
\label{eq:bounded-negative-support-strata}
\mathcal{B}_{\bm{w}}(\Lambda)&:=\left\{I\in\NS(\Lambda)\,\middle\vert\,\Lambda_{I}\text{ is $\bm{w}$-bounded below}\right\},\\
\label{eq:unbounded-negative-support-strata}
\mathcal{U}_{\bm{w}}(\Lambda)&:=\NS(\Lambda)\setminus\mathcal{B}_{\bm{w}}(\Lambda).
\end{align}
For $I\in\mathcal{B}_{\bm{w}}(\Lambda)$, denote the unique minimum by
\[
\bm{v}_{\Lambda,I}:=\min_{\leq_{\bm{w},\Lambda}}\Lambda_{I}.
\]
The families $\mathcal{B}_{\bm{w}}(\Lambda)$ and $\mathcal{U}_{\bm{w}}(\Lambda)$ and the vectors $\bm{v}_{\Lambda,I}$ are independent of the reference vector and of the chosen Gale dual matrix $B$.

\begin{definition}
\label{def:bounded-inclusion-minimal-negative-support-family}
The \emph{bounded inclusion-minimal negative-support family} is
\[
\mathcal{M}in_{\bm{w}}(\Lambda):=\left\{ I\in\mathcal{B}_{\bm{w}}(\Lambda)\,\middle\vert\,\text{there is no }J\in\NS(\Lambda)\text{ with }J\subsetneq I \right\}.
\]
An exponent lattice $\Lambda\in\mathfrak{L}_{\bm{\beta}}$ is called \emph{$\bm{w}$-admissible} if $\mathcal{M}in_{\bm{w}}(\Lambda)\neq\emptyset$.
Put
\begin{equation}
\label{eq:admissible-exponent-lattices}
\mathfrak{L}_{\bm{\beta},\bm{w}}^{\mathrm{adm}}:=\left\{\Lambda\in\mathfrak{L}_{\bm{\beta}}\,\vert\,\mathcal{M}in_{\bm{w}}(\Lambda)\neq\emptyset\right\}.
\end{equation}
\end{definition}

\section{Classwise normalized coefficients and finite directional modules}
\label{sec:classwise-finite-presentation}
Throughout this section, we fix an $L$-generic direction $\bm{w}\in\mathcal{W}_{L}^{\mathrm{gen}}$.
After proving that the Euler equations restrict every ambient formal logarithmic solution to its exponent lattice, we impose the directional condition and derive the finite normalized coefficient system on that lattice.
The coefficient comparison and normalization are exponent-lattice counterparts of the constructions in \cite[Section~4]{Log3}, where the operators are indexed by lattice shifts from a fixed fake exponent rather than directly by pairs of exponents.

\subsection{Euler support reduction and classwise directed solutions}
\label{subsec:ambient-Euler-support-reduction}

\begin{proposition}
\label{prop:ambient-coefficient-criterion-Euler-support-reduction}
Let $\Lambda\in\mathfrak{L}_{\bm{\beta}}$, and let
\[
\varphi_{\Lambda}=\sum_{\bm{\gamma}\in\widetilde{\Lambda}}\bm{x}^{\bm{\gamma}}r_{\Lambda,\bm{\gamma}}(\log\bm{x})\in\mathscr{F}_{\widetilde{\Lambda}}^{\log}.
\]
If $\varphi_{\Lambda}$ is annihilated by the Euler operators $(A\theta_{\bm{x}})_{i}-\beta_{i}$ for $i=1,\dots,d$, then
\begin{equation}
\label{eq:Euler-support-reduction}
\Supp(\varphi_{\Lambda})\subset\widetilde{\Lambda}\cap A_{\C}^{-1}(\bm{\beta})=\Lambda.
\end{equation}
In particular, $\varphi_{\Lambda}\in\mathscr{F}_{\widetilde{\Lambda}}^{\log}[\Lambda]$.
Moreover, its coefficients satisfy
\begin{equation}
\label{eq:raw-Euler-equations}
(A\partial_{\bm{y}})_{i}\mathbin{\bullet}_{\bm{y}}r_{\Lambda,\bm{\gamma}}=0\quad(\bm{\gamma}\in\Lambda,\ i=1,\dots,d).
\end{equation}
\end{proposition}
\begin{proof}
For every $\bm{\gamma}\in\widetilde{\Lambda}$, coefficient comparison for the Euler operator $(A\theta_{\bm{x}})_{i}-\beta_{i}$ gives
\begin{equation}
\label{eq:ambient-Euler-coefficient-equations}
\left((A\bm{\gamma}-\bm{\beta})_{i}+(A\partial_{\bm{y}})_{i}\right)\mathbin{\bullet}_{\bm{y}}r_{\Lambda,\bm{\gamma}}=0\quad(i=1,\dots,d).
\end{equation}
Suppose that $A\bm{\gamma}\neq\bm{\beta}$, and choose $i$ such that $a:=(A\bm{\gamma}-\bm{\beta})_{i}\neq 0$.
Put $q:=(A\partial_{\bm{y}})_{i}$.
The operator $q\mathbin{\bullet}_{\bm{y}}$ lowers polynomial degree and is therefore locally nilpotent on $\C[\bm{y}]$.
Consequently, $a+q$ acts invertibly on $\C[\bm{y}]$.
Indeed, if $p\in\C[\bm{y}]$ has degree at most $m$, then
\[
(a+q)^{-1}\mathbin{\bullet}_{\bm{y}}p=a^{-1}\sum_{k=0}^{m}\left(-a^{-1}q\mathbin{\bullet}_{\bm{y}}\right)^{k}p.
\]
The $i$-th equation in \eqref{eq:ambient-Euler-coefficient-equations} therefore implies $r_{\Lambda,\bm{\gamma}}=0$.
This proves \eqref{eq:Euler-support-reduction}, where the equality on its right follows from \eqref{eq:ambient-components-indexed-by-L-cosets}.
For $\bm{\gamma}\in\Lambda$, we have $A\bm{\gamma}=\bm{\beta}$.
Thus \eqref{eq:ambient-Euler-coefficient-equations} reduces to \eqref{eq:raw-Euler-equations}.
\end{proof}
\begin{definition}
\label{def:w-directed-classwise-formal-solution}
Let $\Lambda\in\mathfrak{L}_{\bm{\beta}}$.
A nonzero formal logarithmic solution $\varphi_{\Lambda}\in\mathscr{F}_{\widetilde{\Lambda}}^{\log}[\Lambda]$ of $H_{A}(\bm{\beta})$ is said to be \emph{$\bm{w}$-directed} if $\Supp(\varphi_{\Lambda})$ has a $\bm{w}$-minimum with respect to $\leq_{\bm{w},\Lambda}$.
The zero solution is also declared to be $\bm{w}$-directed.
\end{definition}
If $\varphi_{\Lambda}$ is nonzero and $\bm{w}$-directed, the unique $\bm{w}$-minimum of its support is called its \emph{lowest exponent} and is denoted by
\[
\lexp_{\bm{w}}(\varphi_{\Lambda}):=\min_{\leq_{\bm{w},\Lambda}}\Supp(\varphi_{\Lambda}).
\]
Define the space of $\bm{w}$-directed classwise formal logarithmic solutions supported on $\Lambda$ by
\begin{equation}
\label{eq:classwise-directed-formal-solution-space}
\Sol_{\bm{w},\Lambda}^{\mathrm{dir}}\bigl(H_{A}(\bm{\beta})\bigr):=\left\{ \varphi_{\Lambda}\in\mathscr{F}_{\widetilde{\Lambda}}^{\log}[\Lambda] \,\middle\vert\, \begin{array}{l}H_{A}(\bm{\beta})\mathbin{\bullet}\varphi_{\Lambda}=0,\\ \varphi_{\Lambda}\text{ is $\bm{w}$-directed} \end{array} \right\}.
\end{equation}
We now fix $\Lambda\in\mathfrak{L}_{\bm{\beta}}$ for the remainder of this section and write
\begin{equation}
\label{eq:classwise-formal-logarithmic-expression}
\varphi=\sum_{\bm{\gamma}\in\Lambda}\bm{x}^{\bm{\gamma}}r_{\bm{\gamma}}(\log\bm{x})\in\mathscr{F}_{\widetilde{\Lambda}}^{\log}[\Lambda].
\end{equation}
For $\bm{\alpha}=(\alpha_{1},\dots,\alpha_{n})^{T}\in\C^{n}$ and $\bm{\nu}=(\nu_{1},\dots,\nu_{n})^{T}\in\N^{n}$, put
\begin{equation}
\label{eq:falling-factorial-differential-operator}
[\partial_{\bm{y}}+\bm{\alpha}]_{\bm{\nu}}:=\prod_{j=1}^{n}\prod_{k=0}^{\nu_{j}-1}(\partial_{y_{j}}+\alpha_{j}-k)\in S,
\end{equation}
where an empty product is understood to be $1$.
For $\bm{\gamma},\bm{\gamma}'\in\Lambda$, define the \emph{raw transit operator} from $\bm{\gamma}$ to $\bm{\gamma}'$ by
\begin{equation}
\label{eq:raw-transit-operator}
d_{\bm{\gamma}'\leftarrow\bm{\gamma}}:=[\partial_{\bm{y}}+\bm{\gamma}]_{(\bm{\gamma}-\bm{\gamma}')_{+}}\in S.
\end{equation}
This definition is valid because $\bm{\gamma}-\bm{\gamma}'\in L\subset\Z^{n}$.
The relation-lattice propagation underlying the following equations is the logarithmic counterpart of the coefficient recurrence encoded by the classical $\Gamma$-series of Gel'fand, Zelevinsky, and Kapranov \cite[Section~1.1, Lemma~1]{GKZ89}; direct pairwise coefficient comparison for the toric operators, in the form established in \cite[Lemma~4.1]{Log3}, gives
\begin{equation}
\label{eq:raw-transit-equations}
d_{\bm{\gamma}'\leftarrow\bm{\gamma}}\mathbin{\bullet}_{\bm{y}}r_{\bm{\gamma}}=d_{\bm{\gamma}\leftarrow\bm{\gamma}'}\mathbin{\bullet}_{\bm{y}}r_{\bm{\gamma}'}\quad(\bm{\gamma},\bm{\gamma}'\in\Lambda).
\end{equation}
Consequently, \eqref{eq:classwise-formal-logarithmic-expression} is annihilated by $H_{A}(\bm{\beta})$ if and only if its coefficients satisfy \eqref{eq:raw-Euler-equations} and \eqref{eq:raw-transit-equations}.
\subsection{The directional normalized coefficient system}
\label{subsec:directed-classwise-solutions-directional-vanishing}
We now normalize the raw coefficient equations on the exponent lattice selected by the Euler equations and incorporate directional vanishing into the resulting finite system.
For the classwise expression \eqref{eq:classwise-formal-logarithmic-expression}, write
\[
\bm{r}=(r_{\bm{\gamma}})_{\bm{\gamma}\in\Lambda}=\Psi_{\widetilde{\Lambda}}(\varphi)\in\mathscr{C}_{\widetilde\Lambda}[\Lambda].
\]
For brevity, write $I_{\bm{\gamma}} := \nsupp(\bm{\gamma})$ for $\bm{\gamma}\in\Lambda$.
For $K\subset\{1,\dots,n\}$, put
\[
\partial_{\bm{y}}^{K}:=\prod_{j\in K}\partial_{y_{j}},\quad\partial_{\bm{y}}^{\emptyset}:=1.
\]

Let $\bm{\gamma},\bm{\gamma}'\in\Lambda$.
Recall that $d_{\bm{\gamma}'\leftarrow\bm{\gamma}} = [\partial_{\bm{y}}+\bm{\gamma}]_{(\bm{\gamma}-\bm{\gamma}')_{+}}$.
By the coordinatewise factorization of the falling factorial operators, the squarefree monomial $\partial_{\bm{y}}^{ I_{\bm{\gamma}'}\setminus I_{\bm{\gamma}} }$ divides $d_{\bm{\gamma}'\leftarrow\bm{\gamma}}$ in $S$.

For $\bm{\gamma},\bm{\gamma}'\in\Lambda$, define
\begin{equation}
\label{eq:reduced-transit-operator}
\widetilde d_{\bm{\gamma}'\leftarrow\bm{\gamma}}:=\frac{d_{\bm{\gamma}'\leftarrow\bm{\gamma}}}{\partial_{\bm{y}}^{I_{\bm{\gamma}'}\setminus I_{\bm{\gamma}}}}\in S.
\end{equation}

The reduced operators satisfy the invertibility and cocycle properties needed for normalization.

\begin{lemma}
\label{lem:reduced-transit-identities}
Let $\bm{\gamma},\bm{\gamma}',\bm{\gamma}''\in\Lambda$.
Then the following statements hold:
\begin{enumerate}[\rm (1)]
\item The action
$\widetilde d_{\bm{\gamma}'\leftarrow\bm{\gamma}} \mathbin{\bullet}_{\bm{y}}: \C[\bm{y}] \longrightarrow \C[\bm{y}]$
is a $\C$-linear automorphism.
\item We have the reduced cocycle identity
\[
\widetilde d_{\bm{\gamma}\leftarrow\bm{\gamma}''}\widetilde d_{\bm{\gamma}''\leftarrow\bm{\gamma}'}\widetilde d_{\bm{\gamma}'\leftarrow\bm{\gamma}}=\widetilde d_{\bm{\gamma}\leftarrow\bm{\gamma}'}\widetilde d_{\bm{\gamma}'\leftarrow\bm{\gamma}''}\widetilde d_{\bm{\gamma}''\leftarrow\bm{\gamma}}.
\]
\end{enumerate}
\end{lemma}

\begin{proof}
The corresponding identities for lattice shifts from a fixed fake exponent are proved in \cite[Lemma~4.2]{Log3}; we include a proof because the present formulation is indexed directly by arbitrary pairs of exponents in $\Lambda$ and requires neither homogeneity nor a chosen fake exponent.
Fix a coordinate $\nu\in[n]$.
The factor $\partial_{y_{\nu}}$ occurs in the $\nu$-th factor of $d_{\bm{\gamma}'\leftarrow\bm{\gamma}}$ if and only if there is an integer $\mu$ with $1\leq\mu\leq\gamma_{\nu}-\gamma'_{\nu}$ and $\gamma_{\nu}-\mu+1=0$.
Because $\gamma_{\nu}-\gamma'_{\nu}\in\Z$, this is equivalent to $\gamma_{\nu}\in\N$ and $\gamma'_{\nu}\in\Z_{<0}$, hence to $\nu\in I_{\bm{\gamma}'}\setminus I_{\bm{\gamma}}$.
Such a factor occurs at most once.
After the factors indexed by $I_{\bm{\gamma}'}\setminus I_{\bm{\gamma}}$ are removed, every remaining affine factor has nonzero constant term.
This proves the factorization in \eqref{eq:reduced-transit-operator} and shows that $\widetilde d_{\bm{\gamma}'\leftarrow\bm{\gamma}}$ has nonzero constant term.
Write $\widetilde d_{\bm{\gamma}'\leftarrow\bm{\gamma}}=a+q$, where $a\in\C^{\times}$ and $q$ has positive degree in the variables $\partial_{y_{j}}$.
On each finite-dimensional space of polynomials of degree at most $m$, the operator $q\mathbin{\bullet}_{\bm{y}}$ is nilpotent, so
\[
\left(\widetilde d_{\bm{\gamma}'\leftarrow\bm{\gamma}}\mathbin{\bullet}_{\bm{y}}\right)^{-1}=a^{-1}\sum_{k=0}^{m}(-a^{-1}q\mathbin{\bullet}_{\bm{y}})^{k}.
\]
These inverses are compatible as $m$ varies, proving that the action on $\C[\bm{y}]$ is an automorphism.
For the cocycle identity, fix $\nu$ and put $a=\gamma_{\nu}$, $b=\gamma'_{\nu}$, and $c=\gamma''_{\nu}$.
The differences $a-b$, $b-c$, and $c-a$ are integers with sum zero.
A direct comparison of the consecutive affine factors in $\partial_{y_{\nu}}$ shows that
\[
d_{\bm{\gamma}\leftarrow\bm{\gamma}''}d_{\bm{\gamma}''\leftarrow\bm{\gamma}'}d_{\bm{\gamma}'\leftarrow\bm{\gamma}}=d_{\bm{\gamma}\leftarrow\bm{\gamma}'}d_{\bm{\gamma}'\leftarrow\bm{\gamma}''}d_{\bm{\gamma}''\leftarrow\bm{\gamma}}
\]
coordinatewise, because the two sides have the same multiset of affine factors.
Taking the product over $\nu$ gives the raw identity.
The squarefree monomial factors on its two sides agree, and cancellation yields the asserted reduced cocycle identity.
\end{proof}

Fix a reference vector $\bm{\lambda}\in\Lambda$.
For a raw coefficient family
$\bm{r} = (r_{\bm{\gamma}})_{\bm{\gamma}\in\Lambda} \in \mathscr{C}_{\widetilde\Lambda}[\Lambda]$,
define its normalized coefficient at $\bm{\gamma}\in\Lambda$ by
\begin{equation}
\label{eq:normalized-coefficient}
c_{\bm{\gamma}}^{(\bm{\lambda})}
:=
\left(
\widetilde d_{\bm{\gamma}\leftarrow\bm{\lambda}}
\mathbin{\bullet}_{\bm{y}}
\right)^{-1}
\left(
\widetilde d_{\bm{\lambda}\leftarrow\bm{\gamma}}
\mathbin{\bullet}_{\bm{y}}
r_{\bm{\gamma}}
\right).
\end{equation}
This is well defined by \cref{lem:reduced-transit-identities}, and the raw coefficient is recovered by
\begin{equation}
\label{eq:raw-coefficient-reconstruction}
r_{\bm{\gamma}}
=
\left(
\widetilde d_{\bm{\lambda}\leftarrow\bm{\gamma}}
\mathbin{\bullet}_{\bm{y}}
\right)^{-1}
\left(
\widetilde d_{\bm{\gamma}\leftarrow\bm{\lambda}}
\mathbin{\bullet}_{\bm{y}}
c_{\bm{\gamma}}^{(\bm{\lambda})}
\right).
\end{equation}

We say that $I\in\NS(\Lambda)$ is \emph{$\bm{w}$-acceptable} if every $J\in\NS(\Lambda)$ satisfying $J\subset I$ belongs to $\mathcal{B}_{\bm{w}}(\Lambda)$.
Denote by $\mathcal{A}_{\bm{w}}(\Lambda)$ the set of $\bm{w}$-acceptable negative supports in $\NS(\Lambda)$.
By definition,
\begin{equation}
\label{eq:unacceptable-support-characterization}
I\notin\mathcal{A}_{\bm{w}}(\Lambda)\iff\text{there exists }J\in\mathcal{U}_{\bm{w}}(\Lambda)\text{ such that }J\subset I.
\end{equation}

\begin{theorem}
\label{thm:classwise-directional-coefficient-characterization}
Fix $\bm{\lambda}\in\Lambda$ and define normalized coefficients by \eqref{eq:normalized-coefficient}.
Let
\[
\varphi=\sum_{\bm{\gamma}\in\Lambda}\bm{x}^{\bm{\gamma}}r_{\bm{\gamma}}(\log\bm{x})\in\mathscr{F}_{\widetilde{\Lambda}}^{\log}[\Lambda]
\]
be a formal logarithmic expression as in \eqref{eq:classwise-formal-logarithmic-expression}.
Then the following conditions are equivalent.
\begin{enumerate}[\rm (1)]
\item We have $\varphi\in\Sol_{\bm{w},\Lambda}^{\mathrm{dir}}\bigl(H_{A}(\bm{\beta})\bigr)$.
\item The normalized coefficient $c_{\bm{\gamma}}^{(\bm{\lambda})}$ depends only on $I_{\bm{\gamma}}=\nsupp(\bm{\gamma})$, hence can be denoted by $c_{\Lambda,I}$ for $I\in\NS(\Lambda)$.
The resulting finite family $(c_{\Lambda,I})_{I\in\NS(\Lambda)}$ satisfies
\begin{equation}
\label{eq:normalized-Euler-equation}
(A\partial_{\bm{y}})_{i}\mathbin{\bullet}_{\bm{y}}c_{\Lambda,I}=0\quad(I\in\NS(\Lambda),\ i=1,\dots,d),
\end{equation}
\begin{equation}
\label{eq:normalized-transit-equation}
\partial_{\bm{y}}^{J\setminus I}\mathbin{\bullet}_{\bm{y}}c_{\Lambda,I}=\partial_{\bm{y}}^{I\setminus J}\mathbin{\bullet}_{\bm{y}}c_{\Lambda,J}\quad(I,J\in\NS(\Lambda)),
\end{equation}
and
\begin{equation}
\label{eq:directional-normalized-vanishing}
c_{\Lambda,I}=0\quad(I\in\NS(\Lambda)\setminus\mathcal{A}_{\bm{w}}(\Lambda)).
\end{equation}
\end{enumerate}
Under these equivalent conditions, \eqref{eq:raw-coefficient-reconstruction} recovers the unique raw coefficient family, and
\begin{equation}
\label{eq:support-as-union-of-nonzero-acceptable-strata}
\Supp(\varphi)=\coprod_{\substack{I\in\mathcal{A}_{\bm{w}}(\Lambda)\\c_{\Lambda,I}\neq 0}}\Lambda_{I}.
\end{equation}
If $\varphi\neq 0$, then
\begin{equation}
\label{eq:lowest-exponent-from-nonzero-acceptable-strata}
\lexp_{\bm{w}}(\varphi)=\min_{\leq_{\bm{w},\Lambda}}\left\{\bm{v}_{\Lambda,I}\,\middle\vert\, I\in\mathcal{A}_{\bm{w}}(\Lambda), c_{\Lambda,I}\neq 0\right\}.
\end{equation}
\end{theorem}
\begin{proof}
By the coefficient comparison preceding \eqref{eq:raw-transit-equations}, it remains to normalize the raw Euler and transit equations.
For $\bm{\gamma},\bm{\gamma}'\in\Lambda$, factor the raw transit operators as
\[
d_{\bm{\gamma}'\leftarrow\bm{\gamma}}=\widetilde d_{\bm{\gamma}'\leftarrow\bm{\gamma}}\partial_{\bm{y}}^{I_{\bm{\gamma}'}\setminus I_{\bm{\gamma}}},\quad d_{\bm{\gamma}\leftarrow\bm{\gamma}'}=\widetilde d_{\bm{\gamma}\leftarrow\bm{\gamma}'}\partial_{\bm{y}}^{I_{\bm{\gamma}}\setminus I_{\bm{\gamma}'}}.
\]
Substituting \eqref{eq:raw-coefficient-reconstruction} and using the reduced cocycle identity shows that the raw transit equation is equivalent to
\[
\partial_{\bm{y}}^{I_{\bm{\gamma}'}\setminus I_{\bm{\gamma}}}\mathbin{\bullet}_{\bm{y}}c_{\bm{\gamma}}^{(\bm{\lambda})}=\partial_{\bm{y}}^{I_{\bm{\gamma}}\setminus I_{\bm{\gamma}'}}\mathbin{\bullet}_{\bm{y}}c_{\bm{\gamma}'}^{(\bm{\lambda})}.
\]
Indeed, the common reduced factor has nonzero constant term and hence acts invertibly by \cref{lem:reduced-transit-identities}.
If $I_{\bm{\gamma}}=I_{\bm{\gamma}'}$, this equality gives
$c_{\bm{\gamma}}^{(\bm{\lambda})}=c_{\bm{\gamma}'}^{(\bm{\lambda})}$.
Thus the normalized coefficient depends only on the negative support, and the last displayed equality becomes \eqref{eq:normalized-transit-equation}.
The normalizing operators have constant coefficients and commute with every $(A\partial_{\bm{y}})_{i}$, so the raw Euler equations are equivalent to \eqref{eq:normalized-Euler-equation}.

For $\bm{\gamma}\in\Lambda_{I}$, the reconstruction formula is a composite of invertible constant-coefficient operators.
Consequently, $r_{\bm{\gamma}}=0$ if and only if $c_{\Lambda,I}=0$, and hence
\[
\Supp(\varphi)=\coprod_{\substack{I\in\NS(\Lambda)\\c_{\Lambda,I}\neq0}} \Lambda_{I}.
\]
Because $\varphi$ is $\bm{w}$-directed, no $\bm{w}$-unbounded stratum can occur in its support: otherwise that entire stratum would lie in $\Supp(\varphi)$ and would contain vectors in $\Lambda$ strictly below the lowest exponent of $\varphi$.
Hence $c_{\Lambda,J}=0$ for every $J\in\mathcal{U}_{\bm{w}}(\Lambda)$.
Let $I\notin\mathcal{A}_{\bm{w}}(\Lambda)$.
By \eqref{eq:unacceptable-support-characterization}, choose $J\in\mathcal{U}_{\bm{w}}(\Lambda)$ with $J\subset I$.
The normalized transit equation for $(J,I)$ becomes
\[
c_{\Lambda,I}=\partial_{\bm{y}}^{I\setminus J}\mathbin{\bullet}_{\bm{y}}c_{\Lambda,J}=0.
\]
This proves \eqref{eq:directional-normalized-vanishing}.
Together with the preceding support identity, it also proves \eqref{eq:support-as-union-of-nonzero-acceptable-strata}.
Thus \textup{(1)} implies \textup{(2)}.

Conversely, assume \textup{(2)} and recover the raw coefficients by \eqref{eq:raw-coefficient-reconstruction}.
Reversing the normalization calculation gives the raw Euler and transit equations, so the coefficient comparison above yields $H_{A}(\bm{\beta})\mathbin{\bullet}\varphi=0$.
By \eqref{eq:directional-normalized-vanishing}, every negative support with a nonzero normalized component is acceptable and therefore bounded.
If the family is zero, then $\varphi=0$ and it is directed by definition.
Otherwise, the finite set $\{I\in\mathcal{A}_{\bm{w}}(\Lambda)\mid c_{\Lambda,I}\neq0\}$ is nonempty.
Since $\mathcal{A}_{\bm{w}}(\Lambda)\subset\mathcal{B}_{\bm{w}}(\Lambda)$, each stratum $\Lambda_{I}$ with $c_{\Lambda,I}\neq0$ has the unique $\bm{w}$-minimum $\bm{v}_{\Lambda,I}$.
Let
\[
\bm{v}_{0}:=\min_{\leq_{\bm{w},\Lambda}}\left\{\bm{v}_{\Lambda,I}\,\middle\vert\,I\in\mathcal{A}_{\bm{w}}(\Lambda), c_{\Lambda,I}\neq 0\right\}.
\]
Choose $I_{0}\in\mathcal{A}_{\bm{w}}(\Lambda)$ such that $c_{\Lambda,I_{0}}\neq 0$ and $\bm{v}_{0}=\bm{v}_{\Lambda,I_{0}}$.
Then $\bm{v}_{0}\in\Lambda_{I_{0}}\subset\Supp(\varphi)$ by \eqref{eq:support-as-union-of-nonzero-acceptable-strata}.
For any $\bm{\gamma}\in\Supp(\varphi)$, the same support formula gives an $I\in\mathcal{A}_{\bm{w}}(\Lambda)$ such that $c_{\Lambda,I}\neq 0$ and $\bm{\gamma}\in\Lambda_{I}$.
Hence
\[
\bm{v}_{0}\leq_{\bm{w},\Lambda}\bm{v}_{\Lambda,I}\leq_{\bm{w},\Lambda}\bm{\gamma}.
\]
Therefore, $\bm{v}_{0}$ is the $\bm{w}$-minimum of $\Supp(\varphi)$.
Thus $\varphi$ is $\bm{w}$-directed, and
\[
\lexp_{\bm{w}}(\varphi)=\bm{v}_{0}=\min_{\leq_{\bm{w},\Lambda}}\left\{\bm{v}_{\Lambda,I}\,\middle\vert\,I\in\mathcal{A}_{\bm{w}}(\Lambda), c_{\Lambda,I}\neq 0\right\}.
\]
Therefore \textup{(2)} implies \textup{(1)}.
Uniqueness of the raw family follows from the invertibility of the reduced transit operators.
\end{proof}

Define the directional normalized coefficient space by
\begin{equation}
\label{eq:classwise-directional-coefficient-space}
\Coeff_{\bm{w}}^{\mathrm{dir}}(\Lambda):=\left\{(c_{\Lambda,I})_{I\in\NS(\Lambda)}\,\middle\vert\,\eqref{eq:normalized-Euler-equation},\ \eqref{eq:normalized-transit-equation},\ \eqref{eq:directional-normalized-vanishing}\text{ hold}\right\}.
\end{equation}
\cref{thm:classwise-directional-coefficient-characterization} gives the normalization isomorphism
\begin{equation}
\label{eq:classwise-directional-normalization-isomorphism}
\Sol_{\bm{w},\Lambda}^{\mathrm{dir}}\bigl(H_{A}(\bm{\beta})\bigr)\xrightarrow{\sim}\Coeff_{\bm{w}}^{\mathrm{dir}}(\Lambda).
\end{equation}

\begin{lemma}
\label{lem:acceptable-supports-contain-bounded-minimal-supports}
Every $K\in\mathcal{A}_{\bm{w}}(\Lambda)$ contains an element of $\mathcal{M}in_{\bm{w}}(\Lambda)$.
Moreover,
\begin{equation}
\label{eq:minimal-support-acceptable-bounded-inclusions}
\mathcal{M}in_{\bm{w}}(\Lambda)\subset\mathcal{A}_{\bm{w}}(\Lambda)\subset\mathcal{B}_{\bm{w}}(\Lambda).
\end{equation}
\end{lemma}
\begin{proof}
Let $K\in\mathcal{A}_{\bm{w}}(\Lambda)$.
The finite nonempty family $\{J\in\NS(\Lambda)\mid J\subset K\}$ has an inclusion-minimal member $I$.
Since $K$ is acceptable, every negative support contained in $K$ is $\bm{w}$-bounded.
Hence $I\in\mathcal{M}in_{\bm{w}}(\Lambda)$.
Now let $I\in\mathcal{M}in_{\bm{w}}(\Lambda)$.
If $J\in\NS(\Lambda)$ satisfies $J\subset I$, then the inclusion-minimality of $I$ gives $J=I$.
Since $I$ is $\bm{w}$-bounded, it follows that $I\in\mathcal{A}_{\bm{w}}(\Lambda)$.
The remaining inclusion follows from the definition of acceptability.
\end{proof}
\subsubsection{The finite directional coefficient module}
\label{subsec:finite-directional-coefficient-module}
The normalized system has one polynomial component for each element of the finite set $\NS(\Lambda)$.
We retain the components indexed by nonminimal acceptable supports because they record the corresponding negative-support strata and may detect their minima as realizable exponents.
Define the finite free $S$-module
\[
F_{\Lambda}:=\bigoplus_{I\in\NS(\Lambda)}S e_{\Lambda,I}^{*}.
\]
Let $U_{\Lambda}^{\mathrm{Eul}}$, $U_{\Lambda}^{\mathrm{tr}}$, and $U_{\Lambda,\bm{w}}^{\mathrm{van}}$ be the submodules generated respectively by
\[
(A\partial_{\bm{y}})_{i} e_{\Lambda,I}^{*}\quad (I\in\NS(\Lambda),\ i=1,\dots,d),
\]
\[
\partial_{\bm{y}}^{J\setminus I}e_{\Lambda,I}^{*}-\partial_{\bm{y}}^{I\setminus J}e_{\Lambda,J}^{*}\quad (I,J\in\NS(\Lambda)),
\]
and
\[
e_{\Lambda,K}^{*}\quad (K\in\NS(\Lambda)\setminus\mathcal{A}_{\bm{w}}(\Lambda)).
\]
\begin{definition}
\label{def:directional-coefficient-module}
The \emph{directional relation submodule} is
\[
U_{\Lambda,\bm{w}}^{\mathrm{dir}}
:=
U_{\Lambda}^{\mathrm{Eul}}
+
U_{\Lambda}^{\mathrm{tr}}
+
U_{\Lambda,\bm{w}}^{\mathrm{van}}
\subset
F_{\Lambda},
\]
and the \emph{finite directional coefficient module} is
\begin{equation}
\label{eq:finite-directional-coefficient-module}
M_{\Lambda,\bm{w}}^{\mathrm{dir}}
:=
F_{\Lambda}/U_{\Lambda,\bm{w}}^{\mathrm{dir}}.
\end{equation}
\end{definition}
Since $\NS(\Lambda)$ is finite, both the generators and the defining Euler, transit, and directional vanishing relations are finite in number; hence $M_{\Lambda,\bm{w}}^{\mathrm{dir}}$ is finitely presented.
Evaluation on the residue classes of the negative-support generators gives a canonical isomorphism
\begin{equation}
\label{eq:directional-coefficient-space-Hom}
\Coeff_{\bm{w}}^{\mathrm{dir}}(\Lambda)\xrightarrow{\sim}\Hom_{S}\left(M_{\Lambda,\bm{w}}^{\mathrm{dir}},\C[\bm{y}]\right).
\end{equation}
Indeed, an $S$-linear map from $F_{\Lambda}$ is determined by the images of the generators $e_{\Lambda,I}^{*}$, and it factors through the quotient precisely when those images satisfy the normalized Euler, transit, and directional vanishing equations.
Combining \eqref{eq:classwise-directional-normalization-isomorphism} and \eqref{eq:directional-coefficient-space-Hom} gives the following corollary.
\begin{corollary}
\label{thm:directed-coefficient-module-presentation}
For every reference vector $\bm{\lambda}\in\Lambda$, normalization induces linear isomorphisms
\begin{equation}
\label{eq:directed-classwise-solution-directional-module-Hom}
\Sol_{\bm{w},\Lambda}^{\mathrm{dir}}\bigl(H_{A}(\bm{\beta})\bigr)\simeq\Coeff_{\bm{w}}^{\mathrm{dir}}(\Lambda)\simeq\Hom_{S}\left(M_{\Lambda,\bm{w}}^{\mathrm{dir}},\C[\bm{y}]\right).
\end{equation}
\end{corollary}
The module and its polynomial coefficient space are independent of the normalization reference; that choice enters only in the conversion from normalized coefficients to raw coefficients and formal logarithmic series.
Minimal supports are needed only to identify a finite generating subfamily and to prove finite length.
Let $K\in\mathcal{A}_{\bm{w}}(\Lambda)$.
By \cref{lem:acceptable-supports-contain-bounded-minimal-supports}, choose $I\in\mathcal{M}in_{\bm{w}}(\Lambda)$ with $I\subset K$.
Then the transit relation gives
\begin{equation}
\label{eq:acceptable-generator-from-minimal-generator}
\overline{e}_{\Lambda,K}^{*}=\partial_{\bm{y}}^{K\setminus I}\overline{e}_{\Lambda,I}^{*}\quad\text{in }M_{\Lambda,\bm{w}}^{\mathrm{dir}}.
\end{equation}
Every unacceptable generator vanishes in the quotient.
Consequently,
\begin{equation}
\label{eq:directional-module-generated-by-minimal-supports}
M_{\Lambda,\bm{w}}^{\mathrm{dir}}=\sum_{I\in\mathcal{M}in_{\bm{w}}(\Lambda)}S\overline{e}_{\Lambda,I}^{*}.
\end{equation}

\subsubsection{Solutions attached to bounded inclusion-minimal negative supports and the exact lattice-selection criterion}
\label{subsec:minimal-support-solutions-exact-selection}
For $I\in\mathcal{M}in_{\bm{w}}(\Lambda)$, define the normalized family $\bm{\varepsilon}_{\Lambda,I}=\bigl((\bm{\varepsilon}_{\Lambda,I})_{K}\bigr)_{K\in\NS(\Lambda)}$ by
\begin{equation}
\label{eq:constant-directional-family-minimal-support}
(\bm{\varepsilon}_{\Lambda,I})_{K}:=
\begin{cases}
1,&K=I,\\
0,&K\neq I.
\end{cases}
\end{equation}
Fix $\bm{\lambda}\in\Lambda$ and set
\[
c_{\bm{\gamma}}^{(\bm{\lambda})} := (\bm{\varepsilon}_{\Lambda,I})_{\nsupp(\bm{\gamma})} \quad (\bm{\gamma}\in\Lambda).
\]
Let $\varphi_{\Lambda,I}^{(\bm{\lambda})}$ be the formal logarithmic expression obtained from this normalized family by inverse normalization \eqref{eq:raw-coefficient-reconstruction} and formal logarithmic substitution \eqref{eq:ambient-formal-logarithmic-substitution}.
\begin{theorem}
\label{thm:minimal-support-solutions-nonvanishing}
For every $I\in\mathcal{M}in_{\bm{w}}(\Lambda)$, the family $\bm{\varepsilon}_{\Lambda,I}$ belongs to $\Coeff_{\bm{w}}^{\mathrm{dir}}(\Lambda)$, and
\[
\varphi_{\Lambda,I}^{(\bm{\lambda})}\in\Sol_{\bm{w},\Lambda}^{\mathrm{dir}}\bigl(H_{A}(\bm{\beta})\bigr),\quad\Supp\left(\varphi_{\Lambda,I}^{(\bm{\lambda})}\right)=\Lambda_{I}.
\]
In particular,
\[
\lexp_{\bm{w}}\left(\varphi_{\Lambda,I}^{(\bm{\lambda})}\right)=\bm{v}_{\Lambda,I}.
\]
Moreover, the following conditions are equivalent:
\begin{enumerate}[\rm (1)]
\item $\mathcal{M}in_{\bm{w}}(\Lambda)\neq\emptyset$.
\item $M_{\Lambda,\bm{w}}^{\mathrm{dir}}\neq 0$.
\item $\Sol_{\bm{w},\Lambda}^{\mathrm{dir}}\bigl(H_{A}(\bm{\beta})\bigr)\neq 0$.
\end{enumerate}
\end{theorem}
\begin{proof}
The Euler equations hold because every component of $\bm{\varepsilon}_{\Lambda,I}$ is constant.
If $K\in\NS(\Lambda)\setminus\{I\}$, then $K\setminus I\neq\emptyset$ by the inclusion-minimality of $I$, and hence $\partial_{\bm{y}}^{K\setminus I} \mathbin{\bullet}_{\bm{y}}1=0$.
Therefore every normalized transit equation involving the unique nonzero component is satisfied, and the equations between zero components are immediate.
Finally, $I\in\mathcal{M}in_{\bm{w}}(\Lambda)\subset\mathcal{A}_{\bm{w}}(\Lambda)$, so the directional vanishing equations also hold.
Thus $\bm{\varepsilon}_{\Lambda,I}\in\Coeff_{\bm{w}}^{\mathrm{dir}}(\Lambda)$.
The assertions about the corresponding series follow from \cref{thm:classwise-directional-coefficient-characterization}.

It remains to prove the equivalence of the three conditions.
The preceding construction proves \textup{(1)}$\Rightarrow$\textup{(3)}.
If \textup{(3)} holds, then the isomorphism in \cref{thm:directed-coefficient-module-presentation} identifies a nonzero directed solution with a nonzero element of $\Hom_{S}\left(M_{\Lambda,\bm{w}}^{\mathrm{dir}},\C[\bm{y}]\right)$.
Hence $M_{\Lambda,\bm{w}}^{\mathrm{dir}}\neq 0$.
Thus \textup{(3)}$\Rightarrow$\textup{(2)}.
Finally, \eqref{eq:directional-module-generated-by-minimal-supports} shows that $M_{\Lambda,\bm{w}}^{\mathrm{dir}}=0$ whenever $\mathcal{M}in_{\bm{w}}(\Lambda)=\emptyset$.
Therefore \textup{(2)}$\Rightarrow$\textup{(1)}.
\end{proof}
\begin{definition}
\label{def:classwise-realizable-exponent}
A vector $\bm{v}\in\Lambda$ is said to be \emph{$\Lambda$-realizable in the direction $\bm{w}$} if there exists a nonzero solution $\varphi\in\Sol_{\bm{w},\Lambda}^{\mathrm{dir}}\bigl(H_{A}(\bm{\beta})\bigr)$ such that $\lexp_{\bm{w}}(\varphi)=\bm{v}$.
The family of all such vectors is denoted by $\Rexp_{\bm{w}}(\Lambda)$.
\end{definition}
For every $I\in\mathcal{M}in_{\bm{w}}(\Lambda)$, the stratum minimum $\bm{v}_{\Lambda,I}$ belongs to $\Rexp_{\bm{w}}(\Lambda)$ by \cref{thm:minimal-support-solutions-nonvanishing}.

\section{Global finiteness of admissible exponent lattices}
\label{sec:global-finiteness}
We now prove that only finitely many exponent lattices contribute nonzero directed formal logarithmic solution spaces.
By \cref{thm:minimal-support-solutions-nonvanishing}, these are precisely the $\bm{w}$-admissible exponent lattices.
The proof first extracts a full-rank integral-coordinate condition from a bounded negative-support stratum and then uses Smith normal form to classify the corresponding exponent lattices by finite saturation quotients.
Retain $r=\rank_{\Z}(L)=n-d$ and the Gale dual matrix $B\in\Mat_{n\times r}(\Z)$ fixed in \eqref{eq:exponent-lattice-basis-matrix}, whose columns form a $\Z$-basis of $L$.
Thus $L=B\Z^{r}$; when $r=0$, the matrix $B$ has no columns and every row submatrix has full column rank.
\subsection{Full coordinate rank forced by a bounded stratum}
\label{subsec:full-coordinate-rank-bounded-stratum}

For a subset $E\subset[n]$, let $B_{E}$ denote the submatrix of $B$ consisting of the rows indexed by $E$.
We first show that the existence of any $\bm{w}$-bounded negative-support stratum forces $B_{E_{\Lambda}}$ to have full column rank.

\begin{lemma}
\label{lem:full-coordinate-rank-forced-by-boundedness}
Let $\Lambda\in\mathfrak{L}_{\bm{\beta}}$ satisfy $\mathcal{B}_{\bm{w}}(\Lambda)\neq\emptyset$.
Then $\rank_{\Q}(B_{E_{\Lambda}})=r$.
Equivalently, the linear map $B_{E_{\Lambda}}:\Q^{r}\to\Q^{E_{\Lambda}}$ is injective.
\end{lemma}

\begin{proof}
Choose $I\in\mathcal{B}_{\bm{w}}(\Lambda)$.
Suppose that $\rank_{\Q}(B_{E_{\Lambda}})<r$.
Since $B_{E_{\Lambda}}$ has integer entries, there exists a nonzero vector $\bm{q}\in\Ker_{\Z}(B_{E_{\Lambda}})$.
Then $\bm{q},-\bm{q}\in R_{\Lambda,I}$.
By \cref{prop:directional-boundedness-stratum-minima}, the functional $\bm{\omega}=B^{T}\bm{w}$ is nonnegative on $R_{\Lambda,I}$.
Hence $\bm{\omega}\cdot\bm{q}\geq 0$ and $-\bm{\omega}\cdot\bm{q}\geq 0$, so $\bm{\omega}\cdot\bm{q}=0$.
Therefore $\bm{w}\cdot B\bm{q}=0$.
Since $B\bm{q}\in L\setminus\{\bm{0}\}$, this contradicts the $L$-genericity of $\bm{w}$.
Thus $\rank_{\Q}(B_{E_{\Lambda}})=r$.
\end{proof}

The contrapositive of \cref{lem:full-coordinate-rank-forced-by-boundedness} shows that
\[
\rank_{\Q}(B_{E_{\Lambda}})<r \quad\Longrightarrow\quad \mathcal{B}_{\bm{w}}(\Lambda)=\emptyset.
\]
Hence $\mathcal{M}in_{\bm{w}}(\Lambda)=\emptyset$, and the generation formula \eqref{eq:directional-module-generated-by-minimal-supports} gives $M_{\Lambda,\bm{w}}^{\mathrm{dir}}=0$.

\subsection{Finiteness under fixed full-rank integrality conditions}
\label{subsec:finiteness-fixed-full-rank-integrality-conditions}

Fix a subset $E\subset[n]$ satisfying $\rank_{\Q}(B_{E})=r$.
We classify the exponent lattices whose coordinates indexed by $E$ are all integral and show that they form a finite family.
Define
\[
\mathfrak{L}_{\bm{\beta}}^{\supset E}:=\left\{ \Lambda\in\mathfrak{L}_{\bm{\beta}}\,\middle\vert\,E\subset E_{\Lambda}\right\}.
\]

For the fixed Gale dual matrix $B$, define
\[
\Gamma_{E} := \left\{ \bm{z}\in\C^{r} \,\middle\vert\, B_{E}\bm{z}\in\Z^{E} \right\}.
\]
Since $B_{E}$ has integer entries, we have $\Z^{r} \subset \Gamma_{E}$.
For a subgroup $N\subset\Z^{E}$, write $\operatorname{Sat}(N):=(\Q N)\cap\Z^{E}$ for its saturation in $\Z^{E}$.

\begin{proposition}
\label{prop:finite-classification-prescribed-integral-coordinates}
Assume that $\rank_{\Q}(B_{E})=r$.
Then $\Gamma_E$ is a full-rank lattice in $\R^{r}$ containing $\Z^{r}$, and
\[
\Gamma_{E}/\Z^{r}\simeq\operatorname{Sat}(B_{E}\Z^{r})/B_{E}\Z^{r}.
\]
If $\mathfrak{L}_{\bm{\beta}}^{\supset E}$ is nonempty and $\Lambda_{0}$ is one of its elements, then any $\bm{\lambda}_{0}\in\Lambda_{0}$ induces a bijection
\[
\Gamma_{E}/\Z^{r}\xrightarrow{\sim}\mathfrak{L}_{\bm{\beta}}^{\supset E},\quad\bm{z}+\Z^{r}\longmapsto\Lambda(\bm{\lambda}_{0}+B\bm{z}).
\]
Consequently,
\[
\left\lvert\mathfrak{L}_{\bm{\beta}}^{\supset E}\right\rvert=\left[\operatorname{Sat}(B_{E}\Z^{r}):B_{E}\Z^{r}\right]<\infty.
\]
\end{proposition}
\begin{proof}
Since $B_{E}$ has full column rank, the complex-linear map $B_{E}:\C^{r}\to\C^{E}$ is injective.
Let $\bm{z}\in\Gamma_{E}$, and write $\bm{z}=\bm{x}+\sqrt{-1}\,\bm{y}$ with $\bm{x},\bm{y}\in\R^{r}$.
Since $B_{E}\bm{z}\in\Z^{E}\subset\R^{E}$, we have $B_{E}\bm{y}=\bm{0}$.
The injectivity of $B_{E}$ gives $\bm{y}=\bm{0}$, and hence $\Gamma_{E}\subset\R^{r}$.

By the Smith normal form of $B_{E}$, the group $\Gamma_{E}$ is a full-rank lattice in $\R^{r}$ containing $\Z^{r}$ with finite index; we now identify the resulting finite quotient.
Since $B_{E}$ has full column rank over $\Q$, it admits a left inverse with rational entries.
Thus, if $\bm{z}\in\Gamma_{E}$, then $B_{E}\bm{z}\in\Z^{E}$ implies $\bm{z}\in\Q^{r}$.
Consequently,
\[
B_{E}\bm{z}\in B_{E}\Q^{r}\cap\Z^{E}=\operatorname{Sat}(B_{E}\Z^{r}).
\]
Therefore the homomorphism
\[
\vartheta_{E}:\Gamma_{E}\longrightarrow\operatorname{Sat}(B_{E}\Z^{r})/B_{E}\Z^{r},\quad\bm{z}\longmapsto B_{E}\bm{z}+B_{E}\Z^{r},
\]
is well defined.
If $\vartheta_{E}(\bm{z})=0$, then $B_{E}\,\bm{z}=B_{E}\,\bm{m}$ for some $\bm{m}\in\Z^{r}$.
The injectivity of $B_{E}$ gives $\bm{z}=\bm{m}\in \Z^{r}$, and hence $\Ker(\vartheta_{E})=\Z^{r}$.
Moreover, if $\bm{u}\in\operatorname{Sat}(B_{E}\Z^{r})$, then $\bm{u}=B_{E}\bm{z}\in\Z^{E}$ for some $\bm{z}\in\Q^{r}$, and hence $\bm{z}\in\Gamma_{E}$.
Thus $\vartheta_{E}$ is surjective, and the first isomorphism theorem gives
\[
\Gamma_{E}/\Z^{r}\simeq\operatorname{Sat}(B_{E}\Z^{r})/B_{E}\Z^{r}.
\]

Assume that $\mathfrak{L}_{\bm{\beta}}^{\supset E}$ is nonempty, choose $\Lambda_{0}=\Lambda(\bm{\lambda}_{0})=\bm{\lambda}_{0}+L\in\mathfrak{L}_{\bm{\beta}}^{\supset E}$.
Since the columns of $B$ form a basis of $\Ker_{\C}(A)$, every point of $A_{\C}^{-1}(\bm{\beta})$ can be written uniquely as $\bm{\lambda}_{0}+B\bm{z}$ for some $\bm{z}\in\C^{r}$.

Let $\bm{z}\in\Gamma_{E}$.
Since $AB=0$, we have
\[
A(\bm{\lambda}_{0}+B\bm{z})=A\bm{\lambda}_{0}=\bm{\beta},
\]
and hence $\bm{\lambda}_{0}+B\bm{z}\in A_{\C}^{-1}(\bm{\beta})$.
Moreover, since $\Lambda_{0}\in\mathfrak{L}_{\bm{\beta}}^{\supset E}$, the coordinates of $\bm{\lambda}_{0}$ indexed by $E$ are integral.
By the definition of $\Gamma_{E}$, $B_{E}\bm{z}\in\Z^{E}$, and therefore $(\bm{\lambda}_{0}+B\bm{z})_{E}=(\bm{\lambda}_{0})_{E}+B_{E}\bm{z}\in\Z^{E}$.
Thus
\[
\Lambda(\bm{\lambda}_{0}+B\bm{z})\in\mathfrak{L}_{\bm{\beta}}^{\supset E}.
\]
Consequently, the assignment $\bm{z} \mapsto\Lambda(\bm{\lambda}_{0}+B\bm{z})$ defines a map from $\Gamma_{E}$ to $\mathfrak{L}_{\bm{\beta}}^{\supset E}$.

For $\bm{z},\bm{z}'\in\Gamma_{E}$, by the injectivity of $B$, we have
\begin{align}
\Lambda(\bm{\lambda}_{0}+B\bm{z})=\Lambda(\bm{\lambda}_{0}+B\bm{z}')&\iff B(\bm{z}-\bm{z}')\in L=B\Z^{r}\\
&\iff \bm{z}-\bm{z}'\in\Z^{r}.
\end{align}
Hence the preceding map induces a well-defined injective map
\[
\Gamma_{E}/\Z^{r}\longrightarrow\mathfrak{L}_{\bm{\beta}}^{\supset E},\quad\bm{z}+\Z^{r}\longmapsto\Lambda(\bm{\lambda}_{0}+B\bm{z}).
\]

To prove surjectivity, let $\Lambda=\Lambda(\bm{\lambda})=\bm{\lambda}+L\in\mathfrak{L}_{\bm{\beta}}^{\supset E}$.
Since $\bm{\lambda},\bm{\lambda}_{0}\in A_{\C}^{-1}(\bm{\beta})$, there exists a unique $\bm{z}\in\C^{r}$ such that $\bm{\lambda}=\bm{\lambda}_{0}+B\bm{z}$.
Both $\bm{\lambda}$ and $\bm{\lambda}_{0}$ have integral coordinates indexed by $E$, and hence
\[
B_{E}\bm{z}=\bm{\lambda}_{E}-(\bm{\lambda}_{0})_{E}\in\Z^{E}.
\]
Thus $\bm{z}\in\Gamma_{E}$, and $\Lambda=\Lambda(\bm{\lambda}_{0}+B\bm{z})$.
Therefore the induced map is surjective and hence bijective.

Combining the two bijections gives
\[
\left\lvert\mathfrak{L}_{\bm{\beta}}^{\supset E}\right\rvert=\left\lvert\Gamma_{E}/\Z^{r}\right\rvert=\left[\operatorname{Sat}(B_{E}\Z^{r}):B_{E}\Z^{r}\right]<\infty.
\]
\end{proof}
The saturation index is computable from the Smith normal form of $B_{E}$.
\subsection{Finiteness of the admissible exponent-lattice family}
\label{subsec:finiteness-admissible-exponent-lattice-family}

We now combine the full-rank obstruction with the preceding fixed-set finiteness theorem.

\begin{theorem}
\label{thm:finiteness-admissible-exponent-lattice-family}
For every $L$-generic direction $\bm{w}\in\mathcal{W}_{L}^{\mathrm{gen}}$, the family
\[
\mathfrak{L}_{\bm{\beta},\bm{w}}^{\mathrm{adm}}=\left\{\Lambda\in \mathfrak{L}_{\bm{\beta}}\,\middle\vert\,\mathcal{M}in_{\bm{w}}(\Lambda)\neq\emptyset\right\}
\]
is finite.
\end{theorem}

\begin{proof}
Let $\Lambda\in\mathfrak{L}_{\bm{\beta},\bm{w}}^{\mathrm{adm}}$.
Then $\emptyset\neq\mathcal{M}in_{\bm{w}}(\Lambda)\subset\mathcal{B}_{\bm{w}}(\Lambda)$.
Hence \cref{lem:full-coordinate-rank-forced-by-boundedness} gives $\rank_{\Q}(B_{E_{\Lambda}})=r$.
Moreover, $\Lambda\in\mathfrak{L}_{\bm{\beta}}^{\supset E_{\Lambda}}$.
Therefore,
\[
\mathfrak{L}_{\bm{\beta},\bm{w}}^{\mathrm{adm}}\subset \bigcup_{\substack{E\subset[n]\\\rank_{\mathbb{Q}}(B_{E})=r}} \mathfrak{L}_{\bm{\beta}}^{\supset E}.
\]
By \cref{prop:finite-classification-prescribed-integral-coordinates}, every family $\mathfrak{L}_{\bm{\beta}}^{\supset E}$ appearing on the right-hand side is finite.
Since $[n]$ has only finitely many subsets, the displayed finite union is finite, and so is $\mathfrak{L}_{\bm{\beta},\bm{w}}^{\mathrm{adm}}$.
\end{proof}
By \cref{thm:minimal-support-solutions-nonvanishing,thm:finiteness-admissible-exponent-lattice-family}, the nonzero classwise directed formal logarithmic solution spaces are indexed by the finite set $\mathfrak{L}_{\bm{\beta},\bm{w}}^{\mathrm{adm}}$.
Accordingly, define the \emph{all-lattice directed formal logarithmic solution space} by
\begin{equation}
\label{eq:all-lattice-classwise-directed-solution-space}
\Sol_{\bm{\beta},\bm{w}}^{\mathrm{dir,all}}\bigl(H_A(\bm{\beta})\bigr):=\bigoplus_{\Lambda\in\mathfrak{L}_{\bm{\beta},\bm{w}}^{\mathrm{adm}}}\Sol_{\bm{w},\Lambda}^{\mathrm{dir}} \bigl(H_A(\bm{\beta})\bigr).
\end{equation}

\section{Local finite length and the all-lattice coefficient module}
\label{sec:local-finite-length}
We prove finite length by associating a local constraint ideal with each bounded inclusion-minimal negative support and showing that the corresponding quotient is a finite-dimensional local $\C$-algebra.
By \cref{thm:minimal-support-solutions-nonvanishing,thm:finiteness-admissible-exponent-lattice-family}, the nonzero classwise directional coefficient modules are indexed by the finite set $\mathfrak{L}_{\bm{\beta},\bm{w}}^{\mathrm{adm}}$.
Accordingly, define the \emph{all-lattice directional coefficient module} by
\begin{equation}
\label{eq:all-lattice-directional-coefficient-module-admissible-sum}
M_{\bm{\beta},\bm{w}}^{\mathrm{dir,all}}:=\bigoplus_{\Lambda\in\mathfrak{L}_{\bm{\beta},\bm{w}}^{\mathrm{adm}}} M_{\Lambda,\bm{w}}^{\mathrm{dir}}.
\end{equation}
Taking the finite direct sum of the classwise isomorphisms in \cref{thm:directed-coefficient-module-presentation} gives a canonical isomorphism
\begin{equation}
\label{eq:all-lattice-directed-solution-directional-module-Hom}
\Sol_{\bm{\beta},\bm{w}}^{\mathrm{dir,all}}\bigl(H_{A}(\bm{\beta})\bigr)\simeq \Hom_{S}\left(M_{\bm{\beta},\bm{w}}^{\mathrm{dir,all}},\C[\bm{y}]\right).
\end{equation}

\subsection{Local constraint ideals and their finite-length quotients}
\label{subsec:minimal-support-constraint-ideals-finite-length-quotients}

Put
\[
\mathfrak{m}:=\left\langle\partial_{y_{1}},\dots,\partial_{y_{n}}\right\rangle\subset S.
\]
At each bounded inclusion-minimal negative support, the local constraint ideal consists of the Euler relations and the boundary monomials induced, through normalized transit equations, by the vanishing of unbounded negative-support components.
Transit relations between distinct bounded inclusion-minimal negative supports are not included.

Define the Euler linear ideal by
\begin{equation}
\label{eq:Euler-linear-ideal-contained-maximal-ideal}
\mathfrak{e}_{A}:=\left\langle(A\partial_{\bm{y}})_{i}:=\sum_{j=1}^{n} a_{ij}\partial_{y_{j}}\,\middle\vert\,i=1,\dots,d\right\rangle\subset\mathfrak{m}.
\end{equation}

Throughout this subsection, we fix $\Lambda\in\mathfrak{L}_{\bm{\beta},\bm{w}}^{\mathrm{adm}}$ and $I\in\mathcal{M}in_{\bm{w}}(\Lambda)$.
For every $J\in\mathcal{U}_{\bm{w}}(\Lambda)$, we have $J\setminus I\neq\emptyset$.
Indeed, otherwise $J\subset I$, and the inclusion-minimality of $I$ in $\NS(\Lambda)$ would give $J=I$, contrary to $J\in\mathcal{U}_{\bm{w}}(\Lambda)$ and $I\in\mathcal{B}_{\bm{w}}(\Lambda)$.

Hence we define a monomial ideal by
\begin{equation}
\label{eq:minimal-support-boundary-ideal-contained-maximal-ideal}
\mathfrak{b}_{\Lambda,\bm{w},I}:=\left\langle\partial_{\bm{y}}^{J\setminus I}\,\middle\vert\,J\in\mathcal{U}_{\bm{w}}(\Lambda)\right\rangle\subset \mathfrak{m},
\end{equation}
and call it the \emph{minimal-support boundary ideal at $I$}.

Moreover, we define the \emph{local constraint ideal at $I$} by
\begin{equation}
\label{eq:local-constraint-ideal-at-I}
\mathfrak{r}_{\Lambda,\bm{w},I}:=\mathfrak{e}_{A}+\mathfrak{b}_{\Lambda,\bm{w},I}\subset\mathfrak{m}.
\end{equation}
Define the corresponding \emph{local constraint quotient} by
\[
R_{\Lambda,\bm{w},I}:= S/\mathfrak{r}_{\Lambda,\bm{w},I}.
\]
Since $\mathfrak{r}_{\Lambda,\bm{w},I}\subset\mathfrak{m}$, the quotient $R_{\Lambda,\bm{w},I}$ is nonzero.

We prove that the local constraint ideal has radical $\mathfrak{m}$ by considering a prime ideal containing it.
The differential variables surviving modulo that prime determine a nonzero relation vector in $L$; translating the stratum minimum along a suitable orientation of this vector eventually reaches an unbounded negative-support stratum, and the associated boundary monomial contradicts primality.

\begin{lemma}
\label{lem:prime-support-kernel-relation}
Let $\mathfrak{p}\subset S$ be a prime ideal satisfying $\mathfrak{e}_{A}\subset\mathfrak{p}$.
Define
\[
F(\mathfrak{p}):=\left\{j\in[n]\,\middle\vert\,\partial_{y_{j}}\notin\mathfrak{p}\right\}.
\]
If $F(\mathfrak{p})\neq\emptyset$, then there exists $\bm{0}\neq\bm{g}\in L$ such that $\supp(\bm{g})\subset F(\mathfrak{p})$.
\end{lemma}

\begin{proof}
Put $F:=F(\mathfrak{p})$.
Let $K:= \operatorname{Frac}(S/\mathfrak{p})$ be the fraction field of the integral domain $S/\mathfrak{p}$.
For $j\in F$, let $\xi_{j} \in K$ be the image of $\partial_{y_{j}}$.
By the definition of $F$, we have $\xi_{j}\neq0$ for $j\in F$.
For $j\notin F$, the image of $\partial_{y_{j}}$ in $S/\mathfrak{p}$ is zero.

Since $\mathfrak{e}_{A}\subset\mathfrak{p}$, the Euler linear forms vanish in $K$.
Therefore $A_{F}\bm{\xi}_{F} = \bm{0}$, where $\bm{\xi}_{F} = (\xi_{j})_{j\in F} \in K^{F}$.
The nonzeroness of $\bm{\xi}_{F}$ yields $\Ker_{K}(A_{F})\neq 0$ and consequently $\rank_{K}(A_{F}) <\lvert F\rvert$.

The entries of $A_{F}$ belong to $\Z\subset\Q$.
The rank of a matrix with rational entries is unchanged after extending the coefficient field from $\Q$ to $K$.
Thus $\rank_{\Q}(A_{F})=\rank_{K}(A_{F})<\lvert F\rvert$.
It follows that $\Ker_{\Q}(A_{F})\neq 0$.
Choose $\bm{0}\neq\bm{h}\in\Ker_{\Q}(A_{F})$.
After multiplying by a common positive denominator, we may assume that $\bm{h}\in\Z^{F}\setminus\{\bm{0}\}$.

Let $\bm{g}\in\Z^{n}\setminus\{\bm{0}\}$ be the zero extension of $\bm{h}$ from $\Z^{F}$ to $\Z^{n}$.
Then $A\bm{g}=A_{F}\bm{h}=\bm{0}$, so $\bm{g}\in\Ker_{\Z}(A)=L$.
Moreover, $\supp(\bm{g}) \subset F$.
\end{proof}

\begin{lemma}
\label{lem:asymptotic-negative-support-relation-direction}
Let $\Lambda\in\mathfrak{L}_{\bm{\beta}}$, $\bm{\gamma}\in\Lambda$, and $\bm{g}\in L\setminus\{\bm{0}\}$.
Then there exist $q_{0}\in\N$ and $J\in\NS(\Lambda)$ such that
\begin{equation}
\label{eq:asymptotic-negative-support-constant}
\nsupp(\bm{\gamma}-q\bm{g})=J\quad (q\geq q_{0}).
\end{equation}
If, in addition, $\bm{w}\cdot\bm{g}>0$, then $J \in \mathcal{U}_{\bm{w}}(\Lambda)$.
\end{lemma}

\begin{proof}
For every $j\in E_{\Lambda}$, the coordinate $\gamma_{j}-qg_{j}$ is an integer for every $q\in\N$.
If $g_{j}>0$, then $\gamma_{j}-qg_{j} \to -\infty$ as $q\to+\infty$.
If $g_{j}<0$, then $\gamma_{j}-qg_{j} \to +\infty$ as $q\to+\infty$.
If $g_{j}=0$, then the coordinate is independent of $q$.
Consequently, in all three cases, the assertion $\gamma_{j}-qg_{j}\in\Z_{<0}$ is eventually constant as $q\to+\infty$.

Since $E_{\Lambda}$ is finite, there exists $q_{0}\in\N$ such that the negative support of $\bm{\gamma}-q\bm{g}$ is independent of $q$ for $q\geq q_{0}$.
Denote this negative support by $J$.
Since $\bm{\gamma}-q\bm{g} \in \Lambda$ for $q\in\N$, we have $J\in\NS(\Lambda)$.
This proves \eqref{eq:asymptotic-negative-support-constant}.

Suppose now that $\bm{w}\cdot\bm{g}>0$.
Then $\bm{\gamma}-q\bm{g}\in\Lambda_{J}$ for $q\geq q_{0}$, and
\[
\bm{w}\cdot \left( (\bm{\gamma}-q\bm{g})-\bm{\gamma} \right) = -q(\bm{w}\cdot\bm{g}) \longrightarrow -\infty
\]
as $q\to+\infty$.
Therefore the stratum $\Lambda_{J}$ is not $\bm{w}$-bounded below.
Hence $J\in\mathcal{U}_{\bm{w}}(\Lambda)$.
\end{proof}

\begin{theorem}
\label{thm:finite-length-minimal-support-constraint-quotient}
Let $\Lambda\in\mathfrak{L}_{\bm{\beta},\bm{w}}^{\mathrm{adm}}$ and $I\in\mathcal{M}in_{\bm{w}}(\Lambda)$.
Then $\sqrt{ \mathfrak{r}_{\Lambda,\bm{w},I} } = \mathfrak{m}$.
Consequently, $R_{\Lambda,\bm{w},I}$ is a nonzero finite-dimensional local $\C$-algebra and a finite-length $S$-module.
\end{theorem}

\begin{proof}
Since $\mathfrak{r}_{\Lambda,\bm{w},I}\subset\mathfrak{m}$ and $\mathfrak{m}$ is radical, we have $\sqrt{\mathfrak{r}_{\Lambda,\bm{w},I}}\subset\mathfrak{m}$.

Let $\mathfrak{p}$ be a prime ideal containing $\mathfrak{r}_{\Lambda,\bm{w},I}$, and put
\[
F:=\left\{j\in[n]\,\middle\vert\,\partial_{y_{j}}\notin\mathfrak{p}\right\}.
\]
If $F=\emptyset$, then we have $\mathfrak{p}=\mathfrak{m}$.
Suppose that $F\neq\emptyset$.
By \cref{lem:prime-support-kernel-relation}, there exists a nonzero vector $\bm{g}\in L$ such that $\supp(\bm{g})\subset F$.
After replacing $\bm{g}$ by $-\bm{g}$ if necessary, we may assume that $\bm{w}\cdot\bm{g}>0$.
Apply \cref{lem:asymptotic-negative-support-relation-direction} to the unique $\bm{w}$-minimum $\bm{v}_{\Lambda,I}\in\Lambda_{I}\subset\Lambda$ and $\bm{g}$.
There exists $J\in\mathcal{U}_{\bm{w}}(\Lambda)$ such that $\nsupp(\bm{v}_{\Lambda,I}-q\bm{g})=J$ for every sufficiently large $q$.

Since $g_{j}=0$ for every $j\notin F$, the coordinates of $\bm{v}_{\Lambda,I}-q\bm{g}$ outside $F$ agree with those of $\bm{v}_{\Lambda,I}$.
Since $\nsupp(\bm{v}_{\Lambda,I})=I$, it follows that $J\setminus F=I\setminus F$.
In particular, $J\setminus I\subset F$.
Moreover, $J\setminus I\neq\emptyset$ because $J\in\mathcal{U}_{\bm{w}}(\Lambda)$ and $I\in\mathcal{M}in_{\bm{w}}(\Lambda)\subset\mathcal{B}_{\bm{w}}(\Lambda)$.

Since $J\in\mathcal{U}_{\bm{w}}(\Lambda)$, the monomial $\partial_{\bm{y}}^{J\setminus I}=\prod_{j\in J\setminus I}\partial_{y_{j}}\in\mathfrak{b}_{\Lambda,\bm{w},I}\subset\mathfrak{r}_{\Lambda,\bm{w},I}\subset\mathfrak{p}$.
On the other hand, since $J\setminus I\subset F$, $\partial_{y_{j}}\notin\mathfrak{p}$ for every $j\in J\setminus I$.
This contradicts the primality of $\mathfrak{p}$.
Therefore $F=\emptyset$, and hence $\mathfrak{p}=\mathfrak{m}$.

Thus $\mathfrak{m}$ is the unique prime ideal containing $\mathfrak{r}_{\Lambda,\bm{w},I}$, and hence we have $\sqrt{\mathfrak{r}_{\Lambda,\bm{w},I}}=\mathfrak{m}$.
It follows that $R_{\Lambda,\bm{w},I}$ is a nonzero finite-dimensional local $\C$-algebra and a finite-length $S$-module.
\end{proof}

\subsection{Finite length of the classwise and all-lattice coefficient modules}
\label{subsec:finite-length-classwise-all-lattice-directional-modules}

Each cyclic submodule generated by a bounded inclusion-minimal negative support is a quotient of the corresponding local constraint quotient, and these cyclic submodules generate the directional coefficient module by \eqref{eq:directional-module-generated-by-minimal-supports}.
Fix $\Lambda\in\mathfrak{L}_{\bm{\beta},\bm{w}}^{\mathrm{adm}}$ and write $\overline{e}_{\Lambda,I}^{*}$ for the class of $e_{\Lambda,I}^{*}$ in $M_{\Lambda,\bm{w}}^{\mathrm{dir}}$.

\begin{theorem}
\label{thm:finite-length-classwise-all-lattice-modules}
For every $\Lambda\in\mathfrak{L}_{\bm{\beta},\bm{w}}^{\mathrm{adm}}$, the module $M_{\Lambda,\bm{w}}^{\mathrm{dir}}$ has finite length and satisfies
\[
\Supp_{S}\left(M_{\Lambda,\bm{w}}^{\mathrm{dir}}\right)\subset\{\mathfrak{m}\}.
\]
The all-lattice module $M_{\bm{\beta},\bm{w}}^{\mathrm{dir,all}}$ also has finite length and satisfies
\[
\Supp_{S}\left(M_{\bm{\beta},\bm{w}}^{\mathrm{dir,all}}\right)\subset\{\mathfrak{m}\}.
\]
Moreover,
\begin{equation}
\label{eq:all-lattice-directional-module-length-sum}
\length_{S} M_{\bm{\beta},\bm{w}}^{\mathrm{dir,all}}=\sum_{\Lambda\in\mathfrak{L}_{\bm{\beta},\bm{w}}^{\mathrm{adm}}} \length_{S} M_{\Lambda,\bm{w}}^{\mathrm{dir}}.
\end{equation}
\end{theorem}
\begin{proof}
Fix $\Lambda\in\mathfrak{L}_{\bm{\beta},\bm{w}}^{\mathrm{adm}}$.
For every $I\in\mathcal{M}in_{\bm{w}}(\Lambda)$, the Euler relations annihilate $\overline{e}_{\Lambda,I}^{*}$.
For every $J\in\mathcal{U}_{\bm{w}}(\Lambda)$, the directional vanishing relation gives $\overline{e}_{\Lambda,J}^{*}=0$, and the normalized transit relation for $(I,J)$ gives $\partial_{\bm{y}}^{J\setminus I}\overline{e}_{\Lambda,I}^{*}=0$.
Consequently, $\mathfrak{r}_{\Lambda,\bm{w},I}\subset\Ann_{S}\left(\overline{e}_{\Lambda,I}^{*}\right)$, and hence there is a surjective homomorphism $R_{\Lambda,\bm{w},I}\twoheadrightarrow S\overline{e}_{\Lambda,I}^{*}$.

By \cref{thm:finite-length-minimal-support-constraint-quotient}, each cyclic module $S\overline{e}_{\Lambda,I}^{*}$ has finite length and satisfies $\Supp_{S}\left(S\overline{e}_{\Lambda,I}^{*}\right)\subset\{\mathfrak{m}\}$, because it is a quotient of $R_{\Lambda,\bm{w},I}$.
Since \eqref{eq:directional-module-generated-by-minimal-supports} expresses $M_{\Lambda,\bm{w}}^{\mathrm{dir}}$ as the sum of finitely many such cyclic modules, $M_{\Lambda,\bm{w}}^{\mathrm{dir}}$ has finite length and satisfies
\[
\Supp_{S}\left( M_{\Lambda,\bm{w}}^{\mathrm{dir}} \right) \subset \{\mathfrak{m}\}.
\]
Finally, \eqref{eq:all-lattice-directional-coefficient-module-admissible-sum} expresses $M_{\bm{\beta},\bm{w}}^{\mathrm{dir,all}}$ as a finite direct sum of the classwise modules above.
Hence it has finite length and satisfies
\[
\Supp_{S}\left( M_{\bm{\beta},\bm{w}}^{\mathrm{dir,all}}\right)\subset\{\mathfrak{m}\}.
\]
Additivity of length on finite direct sums gives \eqref{eq:all-lattice-directional-module-length-sum}.
\end{proof}

\section{Dimension, formal realization, and realizable lowest exponents}
\label{sec:dimension-series-realization}
Finite-length inverse-system duality yields the all-lattice dimension formula, after which we give the explicit formal logarithmic realization and finite criteria for classwise realizable lowest exponents.

\subsection{Finite-length inverse-system duality}
Recall that
\[
\mathfrak{m}=\left\langle\partial_{y_{1}},\dots,\partial_{y_{n}}\right\rangle\subset S,
\]
and that $S$ acts on $\C[\bm{y}]$ by constant-coefficient differentiation.
\begin{proposition}
\label{prop:finite-length-inverse-system-duality}
Let $N$ be a finite-length $S$-module satisfying
\[
\Supp_{S}(N) \subset \{\mathfrak{m}\}.
\]
Then
\begin{equation}
\label{eq:inverse-system-dimension-equals-length}
\dim_{\C}\Hom_{S}\left(N,\C[\bm{y}]\right)=\length_{S}(N).
\end{equation}
In particular, $\Hom_{S}\left(N,\C[\bm{y}]\right)$ is finite-dimensional.
\end{proposition}
\begin{proof}
This is the finite-length case of Macaulay--Matlis duality.
In characteristic zero, the Matlis dual admits a polynomial realization that identifies the divided-power model with ordinary polynomials under constant-coefficient differentiation.
See \cite[Section~21]{Eisenbud}, \cite[Theorem~2.1]{SchulzeTozzo}, and \cite[equations~(2.1)--(2.3)]{SchulzeTozzo}.
Matlis duality preserves composition length, which gives \eqref{eq:inverse-system-dimension-equals-length}.
\end{proof}
\subsection{The global length--dimension theorem}
\label{subsec:global-length-dimension-theorem}

\begin{theorem}
\label{thm:global-length-dimension-formula}
The following quantities are finite and equal:
\begin{equation}
\label{eq:global-length-dimension-theorem}
\begin{aligned}
\dim_{\C}\Sol_{\bm{\beta},\bm{w}}^{\mathrm{dir,all}}\bigl(H_{A}(\bm{\beta})\bigr)&=\length_{S}\left(M_{\bm{\beta},\bm{w}}^{\mathrm{dir,all}}\right)\\
&=\sum_{\Lambda\in\mathfrak{L}_{\bm{\beta},\bm{w}}^{\mathrm{adm}}}\length_{S}\left(M_{\Lambda,\bm{w}}^{\mathrm{dir}}\right).
\end{aligned}
\end{equation}
\end{theorem}
\begin{proof}
By \cref{thm:finite-length-classwise-all-lattice-modules}, the module $M_{\bm{\beta},\bm{w}}^{\mathrm{dir,all}}$ has finite length and support contained in $\{\mathfrak{m}\}$.
Therefore, \eqref{eq:all-lattice-directed-solution-directional-module-Hom} and \cref{prop:finite-length-inverse-system-duality} identify the first two quantities in \eqref{eq:global-length-dimension-theorem}.
The final equality follows from \eqref{eq:all-lattice-directional-module-length-sum}.
\end{proof}
\subsection{Explicit formal logarithmic realization}
\label{subsec:explicit-formal-logarithmic-realization}
Choose a reference vector $\bm{\lambda}_{\Lambda}\in\Lambda$ for every $\Lambda\in\mathfrak{L}_{\bm{\beta},\bm{w}}^{\mathrm{adm}}$, and write $\bm{\lambda}:=(\bm{\lambda}_{\Lambda})_{\Lambda\in\mathfrak{L}_{\bm{\beta},\bm{w}}^{\mathrm{adm}}}$.
For each $\Lambda\in\mathfrak{L}_{\bm{\beta},\bm{w}}^{\mathrm{adm}}$, the isomorphism \eqref{eq:directional-coefficient-space-Hom}, inverse normalization \eqref{eq:raw-coefficient-reconstruction}, and the formal substitution map \eqref{eq:ambient-formal-logarithmic-substitution} define a linear map
\[
\Ser_{\Lambda,\bm{w}}^{\bm{\lambda}_{\Lambda}}:\Hom_{S}\left(M_{\Lambda,\bm{w}}^{\mathrm{dir}},\C[\bm{y}]\right)\longrightarrow\Sol_{\bm{w},\Lambda}^{\mathrm{dir}}\bigl(H_{A}(\bm{\beta})\bigr).
\]
\begin{theorem}
\label{thm:formal-logarithmic-realization}
For every $\Lambda\in\mathfrak{L}_{\bm{\beta},\bm{w}}^{\mathrm{adm}}$, the map $\Ser_{\Lambda,\bm{w}}^{\bm{\lambda}_{\Lambda}}$ is a linear isomorphism.
The finite direct sum of these isomorphisms
\begin{equation}
\label{eq:global-series-realization-map}
\Ser_{\bm{\beta},\bm{w}}^{\bm{\lambda}}:=\bigoplus_{\Lambda\in\mathfrak{L}_{\bm{\beta},\bm{w}}^{\mathrm{adm}}}\Ser_{\Lambda,\bm{w}}^{\bm{\lambda}_{\Lambda}}
\end{equation}
is a linear isomorphism
\[
\Hom_{S}\left(M_{\bm{\beta},\bm{w}}^{\mathrm{dir,all}},\C[\bm{y}]\right)
\xrightarrow{\sim}
\Sol_{\bm{\beta},\bm{w}}^{\mathrm{dir,all}}\bigl(H_{A}(\bm{\beta})\bigr).
\]
Explicitly, let $(c_{\Lambda,I})_{I\in\NS(\Lambda)}$ be the normalized coefficient family corresponding to a homomorphism in $\Hom_{S}\left(M_{\Lambda,\bm{w}}^{\mathrm{dir}},\C[\bm{y}]\right)$.
For $\bm{\gamma}\in\Lambda_{I}$, define
\begin{equation}
\label{eq:direct-raw-coefficient-from-directional-family}
r_{\Lambda,\bm{\gamma}}^{(\bm{\lambda}_{\Lambda})}:=\left(\widetilde{d}_{\bm{\lambda}_{\Lambda}\leftarrow\bm{\gamma}}\mathbin{\bullet}_{\bm{y}}\right)^{-1}\left(\widetilde{d}_{\bm{\gamma}\leftarrow\bm{\lambda}_{\Lambda}}\mathbin{\bullet}_{\bm{y}}c_{\Lambda,I}\right).
\end{equation}
Then the corresponding classwise solution is
\begin{equation}
\label{eq:explicit-realized-logarithmic-series}
\varphi_{\Lambda}:=\Phi_{\widetilde{\Lambda}}\left((r_{\Lambda,\bm{\gamma}}^{(\bm{\lambda}_{\Lambda})})_{\bm{\gamma}\in\Lambda}\right)=\sum_{\bm{\gamma}\in\Lambda}\bm{x}^{\bm{\gamma}}\,r_{\Lambda,\bm{\gamma}}^{(\bm{\lambda}_{\Lambda})}(\log\bm{x}).
\end{equation}
\end{theorem}
\begin{proof}
For a fixed admissible exponent lattice $\Lambda$, evaluation on the module generators gives the isomorphism \eqref{eq:directional-coefficient-space-Hom}.
Inverse normalization is bijective by \eqref{eq:normalized-coefficient} and \eqref{eq:raw-coefficient-reconstruction}, while $\Phi_{\widetilde{\Lambda}}$ is the inverse of the coefficient extraction map $\Psi_{\widetilde{\Lambda}}$.
By \cref{thm:classwise-directional-coefficient-characterization}, their composite is a linear isomorphism onto $\Sol_{\bm{w},\Lambda}^{\mathrm{dir}}\bigl(H_{A}(\bm{\beta})\bigr)$, with explicit formulas \eqref{eq:direct-raw-coefficient-from-directional-family} and \eqref{eq:explicit-realized-logarithmic-series}.
The global assertion follows by taking the finite direct sum over $\mathfrak{L}_{\bm{\beta},\bm{w}}^{\mathrm{adm}}$ and using \eqref{eq:all-lattice-directional-coefficient-module-admissible-sum}.
\end{proof}

\subsection{A finite criterion for classwise realizable exponents}
\label{subsec:finite-criterion-all-classwise-realizable-exponents}
Fix $\Lambda\in\mathfrak{L}_{\bm{\beta}}$.
For $K\in\mathcal{A}_{\bm{w}}(\Lambda)$, define the coordinate projection
\[
\pi_{\Lambda,K}:\Coeff_{\bm{w}}^{\mathrm{dir}}(\Lambda)\longrightarrow\C[\bm{y}],\quad (c_{\Lambda,J})_{J\in\NS(\Lambda)}\longmapsto c_{\Lambda,K}.
\]
The finite set $\{\bm{v}_{\Lambda,K}\mid K\in\mathcal{A}_{\bm{w}}(\Lambda)\}$ is totally ordered by $\leq_{\bm{w},\Lambda}$.
For $K\in\mathcal{A}_{\bm{w}}(\Lambda)$, put
\[
\mathcal{A}_{\bm{w}}^{<K}(\Lambda):=\left\{J\in\mathcal{A}_{\bm{w}}(\Lambda)\,\middle\vert\,\bm{v}_{\Lambda,J}<_{\bm{w},\Lambda}\bm{v}_{\Lambda,K}\right\}
\]
and
\[
\pi_{\Lambda,<K}:=\bigoplus_{J\in\mathcal{A}_{\bm{w}}^{<K}(\Lambda)}\pi_{\Lambda,J}.
\]
If $\mathcal{A}_{\bm{w}}^{<K}(\Lambda)=\emptyset$, the target of $\pi_{\Lambda,<K}$ is the zero vector space and $\pi_{\Lambda,<K}=0$.
\begin{theorem}
\label{thm:finite-criterion-all-classwise-realizable-exponents}
Let $K\in\mathcal{A}_{\bm{w}}(\Lambda)$.
Then the following conditions are equivalent.
\begin{enumerate}[\rm (1)]
\item $\bm{v}_{\Lambda,K}\in\Rexp_{\bm{w}}(\Lambda)$.
\item There exists $\bm{c}=(c_{\Lambda,J})_{J\in\NS(\Lambda)}\in\Coeff_{\bm{w}}^{\mathrm{dir}}(\Lambda)$ such that
\[
c_{\Lambda,J}=0\quad\left(J\in\mathcal{A}_{\bm{w}}^{<K}(\Lambda)\right),\quad c_{\Lambda,K}\neq 0.
\]
\item We have
\begin{equation}
\label{eq:realizable-coordinate-kernel-noninclusion}
\Ker(\pi_{\Lambda,<K})\not\subset\Ker(\pi_{\Lambda,K}).
\end{equation}
\end{enumerate}
Consequently,
\begin{equation}
\label{eq:all-classwise-realizable-exponents-finite-criterion}
\Rexp_{\bm{w}}(\Lambda)=\left\{\bm{v}_{\Lambda,K}\,\middle\vert\,K\in\mathcal{A}_{\bm{w}}(\Lambda),\ \Ker(\pi_{\Lambda,<K})\not\subset\Ker(\pi_{\Lambda,K})\right\}.
\end{equation}
\end{theorem}
\begin{proof}
Let $\bm{c}\in\Coeff_{\bm{w}}^{\mathrm{dir}}(\Lambda)$ be nonzero, and let $\varphi$ be its corresponding classwise solution.
By \eqref{eq:support-as-union-of-nonzero-acceptable-strata}, its support is the disjoint union of the strata $\Lambda_{J}$ for which $c_{\Lambda,J}\neq 0$.
By \eqref{eq:lowest-exponent-from-nonzero-acceptable-strata}, its lowest exponent is the least element of the finite set $\left\{\bm{v}_{\Lambda,J}\,\middle\vert\,c_{\Lambda,J}\neq 0\right\}$.
Therefore $\lexp_{\bm{w}}(\varphi)=\bm{v}_{\Lambda,K}$ if and only if $c_{\Lambda,K}\neq0$ and $c_{\Lambda,J}=0$ for every $J\in\mathcal{A}_{\bm{w}}^{<K}(\Lambda)$.
This proves the equivalence of \textup{(1)} and \textup{(2)}.
Condition \textup{(2)} states precisely that there exists an element of $\Ker(\pi_{\Lambda,<K})$ that does not belong to $\Ker(\pi_{\Lambda,K})$.
Hence \textup{(2)} and \textup{(3)} are equivalent.
Applying this equivalence to every $K\in\mathcal{A}_{\bm{w}}(\Lambda)$ gives \eqref{eq:all-classwise-realizable-exponents-finite-criterion}.
\end{proof}

\subsection{A polynomial colon-ideal criterion for classwise realizability}
\label{subsec:polynomial-colon-ideal-criterion-classwise-realizability}
We reformulate the kernel noninclusion in \eqref{eq:realizable-coordinate-kernel-noninclusion} as a polynomial colon-ideal condition in $S$.
Fix $\Lambda\in\mathfrak{L}_{\bm{\beta}}$ and $K\in\mathcal{A}_{\bm{w}}(\Lambda)$.
Recall that
\[
F_{\Lambda}=\bigoplus_{I\in\NS(\Lambda)}S e_{\Lambda,I}^{*}
\]
and that $U_{\Lambda,\bm{w}}^{\mathrm{dir}}\subset F_{\Lambda}$ is the directional relation submodule.
Define
\begin{equation}
\label{eq:earlier-vanishing-relation-submodule}
U_{\Lambda,\bm{w}}^{<K}:=U_{\Lambda,\bm{w}}^{\mathrm{dir}}+\sum_{J\in\mathcal{A}_{\bm{w}}^{<K}(\Lambda)}S e_{\Lambda,J}^{*}\subset F_{\Lambda}.
\end{equation}
Thus the additional generators in \eqref{eq:earlier-vanishing-relation-submodule} force all acceptable components whose stratum minima precede $\bm{v}_{\Lambda,K}$ with respect to $\leq_{\bm{w}}$ to vanish.

Define an $S$-linear monomial map
\begin{equation}
\label{eq:classwise-monomial-elimination-map}
\mu_{\Lambda}:
F_{\Lambda}\longrightarrow S,\quad \mu_{\Lambda}(e_{\Lambda,I}^{*}):=\partial_{\bm{y}}^{I}\quad (I\in\NS(\Lambda)).
\end{equation}
Put
\begin{equation}
\label{eq:classwise-negative-support-monomial-ideal}
\mathfrak{M}_{\Lambda}:=\left\langle\partial_{\bm{y}}^{I}\,\middle\vert\,I\in\NS(\Lambda)\right\rangle\subset S.
\end{equation}
For $K\in\mathcal{A}_{\bm{w}}(\Lambda)$, put
\begin{equation}
\label{eq:earlier-or-unacceptable-support-family}
\mathcal{V}_{\Lambda,\bm{w}}^{<K}:=\left(\NS(\Lambda)\setminus\mathcal{A}_{\bm{w}}(\Lambda)\right)\cup \mathcal{A}_{\bm{w}}^{<K}(\Lambda),
\end{equation}
and define
\begin{align}
\label{eq:earlier-or-unacceptable-monomial-ideal}
\mathfrak{P}_{\Lambda,\bm{w}}^{<K}
&:=
\left\langle
\partial_{\bm{y}}^{J}
\,\middle\vert\,
J\in\mathcal{V}_{\Lambda,\bm{w}}^{<K}
\right\rangle,
\\
\label{eq:classwise-realizability-elimination-ideal}
\mathfrak{Q}_{\Lambda,\bm{w}}^{<K}
&:=
\mathfrak{e}_{A}\mathfrak{M}_{\Lambda}
+
\mathfrak{P}_{\Lambda,\bm{w}}^{<K}
\subset
S.
\end{align}
The empty generating family in \eqref{eq:earlier-or-unacceptable-monomial-ideal} is understood to generate the zero ideal.

\begin{lemma}
\label{lem:monomial-elimination-earlier-vanishing-system}
The map $\mu_{\Lambda}$ in \eqref{eq:classwise-monomial-elimination-map} satisfies
\begin{equation}
\label{eq:kernel-classwise-monomial-elimination-map}
\Ker(\mu_{\Lambda})=U_{\Lambda}^{\mathrm{tr}}.
\end{equation}
Moreover,
\begin{equation}
\label{eq:image-earlier-vanishing-relation-submodule}
\mu_{\Lambda}\left(U_{\Lambda,\bm{w}}^{<K}\right)=\mathfrak{Q}_{\Lambda,\bm{w}}^{<K}
\end{equation}
and
\begin{equation}
\label{eq:preimage-realizability-elimination-ideal}
\mu_{\Lambda}^{-1}\left(\mathfrak{Q}_{\Lambda,\bm{w}}^{<K}\right)=U_{\Lambda,\bm{w}}^{<K}.
\end{equation}
\end{lemma}
\begin{proof}
The image of $\mu_{\Lambda}$ is the monomial ideal $\mathfrak{M}_{\Lambda}$ generated by the squarefree monomials $\partial_{\bm{y}}^{I}$ with $I\in\NS(\Lambda)$.
Hence $\Ker(\mu_{\Lambda})$ is the first syzygy module of this monomial generating family.
The syzygy module of a finite family of monomials is generated by the pairwise least-common-multiple syzygies.
For $I,J\in\NS(\Lambda)$, the corresponding syzygy is
\[
\frac{\partial_{\bm{y}}^{I\cup J}}{\partial_{\bm{y}}^{I}}e_{\Lambda,I}^{*}-\frac{\partial_{\bm{y}}^{I\cup J}}{\partial_{\bm{y}}^{J}}e_{\Lambda,J}^{*}=\partial_{\bm{y}}^{J\setminus I}e_{\Lambda,I}^{*}-\partial_{\bm{y}}^{I\setminus J}e_{\Lambda,J}^{*}.
\]
These are precisely the generators of $U_{\Lambda}^{\mathrm{tr}}$, which proves \eqref{eq:kernel-classwise-monomial-elimination-map}.

For $I\in\NS(\Lambda)$ and $i=1,\dots,d$, we have
\[
\mu_{\Lambda}\left((A\partial_{\bm{y}})_{i}e_{\Lambda,I}^{*}\right)=(A\partial_{\bm{y}})_{i}\partial_{\bm{y}}^{I}.
\]
These images generate $\mathfrak{e}_{A}\mathfrak{M}_{\Lambda}$.
For every $J\in\NS(\Lambda)\setminus\mathcal{A}_{\bm{w}}(\Lambda)$, the directional vanishing generator $e_{\Lambda,J}^{*}$ is mapped to $\partial_{\bm{y}}^{J}$.
For every $J\in\mathcal{A}_{\bm{w}}^{<K}(\Lambda)$, the additional generator $e_{\Lambda,J}^{*}$ in \eqref{eq:earlier-vanishing-relation-submodule} is also mapped to $\partial_{\bm{y}}^{J}$.
Finally, \eqref{eq:kernel-classwise-monomial-elimination-map} gives $\mu_{\Lambda}(U_{\Lambda}^{\mathrm{tr}})=0$.
These observations prove \eqref{eq:image-earlier-vanishing-relation-submodule}.

The inclusion from right to left in \eqref{eq:preimage-realizability-elimination-ideal} follows from \eqref{eq:image-earlier-vanishing-relation-submodule}.
Conversely, let $\eta\in F_{\Lambda}$ satisfy $\mu_{\Lambda}(\eta)\in\mathfrak{Q}_{\Lambda,\bm{w}}^{<K}$.
By \eqref{eq:image-earlier-vanishing-relation-submodule}, there exists $\eta'\in U_{\Lambda,\bm{w}}^{<K}$ such that $\mu_{\Lambda}(\eta')=\mu_{\Lambda}(\eta)$.
Then
\[
\eta-\eta'\in\Ker(\mu_{\Lambda})=U_{\Lambda}^{\mathrm{tr}}\subset U_{\Lambda,\bm{w}}^{<K},
\]
and hence $\eta\in U_{\Lambda,\bm{w}}^{<K}$.
This proves \eqref{eq:preimage-realizability-elimination-ideal}.
\end{proof}

Define the $K$-th earlier-vanishing coefficient space by
\begin{equation}
\label{eq:Kth-earlier-vanishing-coefficient-space}
\mathcal{C}_{\Lambda,\bm{w},K}^{<}:=\pi_{\Lambda,K}\left(\Ker(\pi_{\Lambda,<K})\right)\subset\C[\bm{y}].
\end{equation}
Thus $\mathcal{C}_{\Lambda,\bm{w},K}^{<}$ consists of all normalized $K$-components that occur in directional coefficient families whose earlier acceptable components vanish.

\begin{theorem}
\label{thm:polynomial-colon-ideal-criterion-classwise-realizability}
Let $\Lambda\in\mathfrak{L}_{\bm{\beta}}$ and $K\in\mathcal{A}_{\bm{w}}(\Lambda)$.
Then
\begin{equation}
\label{eq:earlier-vanishing-coefficient-space-colon-inverse-system}
\mathcal{C}_{\Lambda,\bm{w},K}^{<}=\Sol_{\C[\bm{y}]}\left(\mathfrak{Q}_{\Lambda,\bm{w}}^{<K}:\partial_{\bm{y}}^{K}\right).
\end{equation}
Equivalently,
\begin{equation}
\label{eq:annihilator-earlier-vanishing-Kth-component}
\Ann_{S}\left(\mathcal{C}_{\Lambda,\bm{w},K}^{<}\right)=\mathfrak{Q}_{\Lambda,\bm{w}}^{<K}:\partial_{\bm{y}}^{K}.
\end{equation}
Moreover, the following conditions are equivalent.
\begin{enumerate}[\rm (1)]
\item $\bm{v}_{\Lambda,K}\in\Rexp_{\bm{w}}(\Lambda)$.
\item $\mathcal{C}_{\Lambda,\bm{w},K}^{<}\neq0$.
\item
\begin{equation}
\label{eq:proper-colon-ideal-realizability-criterion}
\mathfrak{Q}_{\Lambda,\bm{w}}^{<K}:
\partial_{\bm{y}}^{K}
\neq
S.
\end{equation}
\end{enumerate}
Consequently,
\begin{equation}
\label{eq:classwise-realizable-exponents-colon-criterion}
\Rexp_{\bm{w}}(\Lambda)
=
\left\{
\bm{v}_{\Lambda,K}
\,\middle\vert\,
K\in\mathcal{A}_{\bm{w}}(\Lambda),\
\mathfrak{Q}_{\Lambda,\bm{w}}^{<K}:
\partial_{\bm{y}}^{K}
\neq S
\right\}.
\end{equation}
\end{theorem}
\begin{proof}
Put $M_{\Lambda,\bm{w}}^{<K}:=F_{\Lambda}/U_{\Lambda,\bm{w}}^{<K}$, and denote the class of $e_{\Lambda,K}^{*}$ in this quotient by $\overline{e}_{\Lambda,K}^{<K}$.
Since $U_{\Lambda,\bm{w}}^{\mathrm{dir}}\subset U_{\Lambda,\bm{w}}^{<K}$, the module $M_{\Lambda,\bm{w}}^{<K}$ is a quotient of $M_{\Lambda,\bm{w}}^{\mathrm{dir}}$.
By \cref{lem:acceptable-supports-contain-bounded-minimal-supports}, the acceptable support $K$ contains an element of $\mathcal{M}in_{\bm{w}}(\Lambda)$.
Hence $\Lambda$ is $\bm{w}$-admissible.
By \cref{thm:finite-length-classwise-all-lattice-modules}, the module $M_{\Lambda,\bm{w}}^{\mathrm{dir}}$ has finite length and satisfies
\[
\Supp_{S}\left(M_{\Lambda,\bm{w}}^{\mathrm{dir}}\right)\subset\{\mathfrak{m}\}.
\]
Since $M_{\Lambda,\bm{w}}^{<K}$ is a quotient of $M_{\Lambda,\bm{w}}^{\mathrm{dir}}$, it also has finite length and satisfies
\[
\Supp_{S}\left(M_{\Lambda,\bm{w}}^{<K}\right)\subset\{\mathfrak{m}\}.
\]

By the definition of $U_{\Lambda,\bm{w}}^{<K}$, evaluation on the residue classes of the generators gives a canonical identification
\[
\Ker(\pi_{\Lambda,<K})\simeq\Hom_{S}\left(M_{\Lambda,\bm{w}}^{<K},\C[\bm{y}]\right).
\]
Under this identification, the $K$-th component of the normalized coefficient family corresponding to a homomorphism $\psi$ is $c_{\Lambda,K}=\psi\left(\overline{e}_{\Lambda,K}^{<K}\right)$.
Consequently,
\[
\mathcal{C}_{\Lambda,\bm{w},K}^{<}=\left\{\psi\left(\overline{e}_{\Lambda,K}^{<K}\right)\,\middle\vert\,\psi\in\Hom_{S}\left(M_{\Lambda,\bm{w}}^{<K},\C[\bm{y}]\right)\right\}.
\]

Consider the short exact sequence
\[
0\longrightarrow S\overline{e}_{\Lambda,K}^{<K}\xrightarrow{\iota} M_{\Lambda,\bm{w}}^{<K} \longrightarrow M_{\Lambda,\bm{w}}^{<K}/S\overline{e}_{\Lambda,K}^{<K} \longrightarrow 0.
\]
All three modules in this sequence have finite length, and their supports are contained in $\{\mathfrak{m}\}$.
Under the characteristic-zero polynomial realization of Macaulay--Matlis duality, the functor $\Hom_{S}(-,\C[\bm{y}])$ is exact on finite-length $S$-modules supported at $\mathfrak{m}$; see \citep[Section~21]{Eisenbud} and \citep[Theorem~2.1 and equations~(2.1)--(2.3)]{SchulzeTozzo}.
Applying this functor to the preceding short exact sequence therefore gives an exact sequence
\begin{align*}
0 & \longrightarrow \Hom_{S}\left(M_{\Lambda,\bm{w}}^{<K}/S\overline{e}_{\Lambda,K}^{<K},\C[\bm{y}]\right)\longrightarrow \Hom_{S}\left(M_{\Lambda,\bm{w}}^{<K},\C[\bm{y}]\right)\\
 &\xlongrightarrow{\iota^{*}} \Hom_{S}\left(S\overline{e}_{\Lambda,K}^{<K},\C[\bm{y}]\right)\longrightarrow 0,
\end{align*}
where $\iota^{*}(\psi)=\left.\psi\right|_{S\overline{e}_{\Lambda,K}^{<K}}$.
In particular, the restriction map $\iota^{*}$ is surjective.

Every homomorphism $\eta:S\overline{e}_{\Lambda,K}^{<K}\to\C[\bm{y}]$ is uniquely determined by the value $\eta(\overline{e}_{\Lambda,K}^{<K})$.
The preceding description of $\mathcal{C}_{\Lambda,\bm{w},K}^{<}$ and the surjectivity of $\iota^{*}$ therefore show that evaluation at the cyclic generator induces a canonical linear isomorphism
\begin{equation}
\label{eq:earlier-vanishing-component-as-cyclic-dual}
\Hom_{S}\left(S\overline{e}_{\Lambda,K}^{<K},\C[\bm{y}]\right)\xrightarrow{\sim}\mathcal{C}_{\Lambda,\bm{w},K}^{<},\quad\eta\longmapsto\eta\left(\overline{e}_{\Lambda,K}^{<K}\right).
\end{equation}

The annihilator of $\overline{e}_{\Lambda,K}^{<K}$ is the module colon ideal
\[
U_{\Lambda,\bm{w}}^{<K}:e_{\Lambda,K}^{*}:=\left\{f\in S\,\middle\vert\,f e_{\Lambda,K}^{*}\in U_{\Lambda,\bm{w}}^{<K}\right\}.
\]
By \eqref{eq:preimage-realizability-elimination-ideal}, for every $f\in S$, we have
\begin{align*}
f e_{\Lambda,K}^{*}\in U_{\Lambda,\bm{w}}^{<K}&\iff \mu_{\Lambda}\left(f e_{\Lambda,K}^{*}\right)\in\mathfrak{Q}_{\Lambda,\bm{w}}^{<K}\\
&\iff f\partial_{\bm{y}}^{K}\in\mathfrak{Q}_{\Lambda,\bm{w}}^{<K}\iff f\in\mathfrak{Q}_{\Lambda,\bm{w}}^{<K}:\partial_{\bm{y}}^{K}.
\end{align*}
Hence $U_{\Lambda,\bm{w}}^{<K}:e_{\Lambda,K}^{*}=\mathfrak{Q}_{\Lambda,\bm{w}}^{<K}:\partial_{\bm{y}}^{K}$.

The standard presentation of a cyclic module by the annihilator of its generator therefore gives $S\overline{e}_{\Lambda,K}^{<K}\simeq S/\left(\mathfrak{Q}_{\Lambda,\bm{w}}^{<K}:\partial_{\bm{y}}^{K}\right)$, where the isomorphism is induced by $f+\left(\mathfrak{Q}_{\Lambda,\bm{w}}^{<K}:\partial_{\bm{y}}^{K}\right)\mapsto f\overline{e}_{\Lambda,K}^{<K}$.

Evaluation at the residue class of $1$ induces a canonical linear isomorphism
\[
\Hom_{S}\left(S/(\mathfrak{Q}_{\Lambda,\bm{w}}^{<K}:\partial_{\bm{y}}^{K}),\C[\bm{y}]\right)\xrightarrow{\sim}\Sol_{\C[\bm{y}]}\left(\mathfrak{Q}_{\Lambda,\bm{w}}^{<K}:\partial_{\bm{y}}^{K}\right),\quad \eta\longmapsto\eta(\overline{1}).
\]
Indeed, $f\mathbin{\bullet}_{\bm{y}}\eta(\overline{1})=\eta(\overline{f})=0$ for any $f\in\mathfrak{Q}_{\Lambda,\bm{w}}^{<K}:\partial_{\bm{y}}^{K}$.
Conversely, every $p\in\Sol_{\C[\bm{y}]}\left(\mathfrak{Q}_{\Lambda,\bm{w}}^{<K}:\partial_{\bm{y}}^{K}\right)$ defines an $S$-linear homomorphism by $\overline{f}\mapsto f\mathbin{\bullet}_{\bm{y}}p$.
Combining this identification with \eqref{eq:earlier-vanishing-component-as-cyclic-dual} gives \eqref{eq:earlier-vanishing-coefficient-space-colon-inverse-system}.

If $\mathfrak{Q}_{\Lambda,\bm{w}}^{<K}:\partial_{\bm{y}}^{K}=S$, then both sides of \eqref{eq:annihilator-earlier-vanishing-Kth-component} are equal to $S$.
Suppose that $\mathfrak{Q}_{\Lambda,\bm{w}}^{<K}:\partial_{\bm{y}}^{K}\neq S$.
Since $S/\left(\mathfrak{Q}_{\Lambda,\bm{w}}^{<K}:\partial_{\bm{y}}^{K}\right)\simeq S\overline{e}_{\Lambda,K}^{<K}$ is a nonzero finite-length module whose support is contained in $\{\mathfrak{m}\}$, the ideal $\mathfrak{Q}_{\Lambda,\bm{w}}^{<K}:\partial_{\bm{y}}^{K}$ is $\mathfrak{m}$-primary.
The finite-length double-annihilator property gives
\[
\Ann_{S}\left(\Sol_{\C[\bm{y}]}\left(\mathfrak{Q}_{\Lambda,\bm{w}}^{<K}:\partial_{\bm{y}}^{K}\right)\right)=\mathfrak{Q}_{\Lambda,\bm{w}}^{<K}:\partial_{\bm{y}}^{K}.
\]
Together with \eqref{eq:earlier-vanishing-coefficient-space-colon-inverse-system}, this yields \eqref{eq:annihilator-earlier-vanishing-Kth-component}.

The equivalence of \textup{(1)} and \textup{(2)} follows from \cref{thm:finite-criterion-all-classwise-realizable-exponents}.
By \eqref{eq:earlier-vanishing-coefficient-space-colon-inverse-system} and \cref{prop:finite-length-inverse-system-duality}, we have
\[
\mathcal{C}_{\Lambda,\bm{w},K}^{<}\neq 0 \iff S/\left(\mathfrak{Q}_{\Lambda,\bm{w}}^{<K}:\partial_{\bm{y}}^{K}\right)\neq 0.
\]
The latter condition is equivalent to $\mathfrak{Q}_{\Lambda,\bm{w}}^{<K}:\partial_{\bm{y}}^{K}\neq S$.
Thus \textup{(2)} and \textup{(3)} are equivalent.
Applying the resulting criterion to every $K\in\mathcal{A}_{\bm{w}}(\Lambda)$ proves \eqref{eq:classwise-realizable-exponents-colon-criterion}.
\end{proof}

\begin{remark}
\label{rem:computability-colon-criterion-nonhomogeneous-case}
The ideal $\mathfrak{Q}_{\Lambda,\bm{w}}^{<K}$ is generated explicitly by the Euler linear forms and finitely many squarefree monomials determined by $\NS(\Lambda)$, $\mathcal{A}_{\bm{w}}(\Lambda)$, and the ordering of the stratum minima.
Hence the colon ideal in \eqref{eq:proper-colon-ideal-realizability-criterion} is computable by standard polynomial ideal operations.

The role of this colon ideal differs from that of the colon ideal constructed in \cite[Theorem~4.7]{Log3}.
There, one fixes a fake exponent and an ordered negative-support family, eliminates the components of the corresponding finite normalized coefficient module, and obtains a single colon ideal whose polynomial inverse system is the entire space of logarithmic polynomials that occur as coefficients at that fixed fake exponent.
In the present setting, by contrast, the basic object is an exponent lattice $\Lambda$, and for each acceptable negative support $K$ we impose the additional vanishing of all acceptable components whose stratum minima strictly precede $\bm{v}_{\Lambda,K}$.
The resulting $K$-dependent colon ideal $\mathfrak{Q}_{\Lambda,\bm{w}}^{<K}:\partial_{\bm{y}}^{K}$ therefore describes the possible $K$-components under these earlier-vanishing conditions and, in particular, determines whether $\bm{v}_{\Lambda,K}$ is realizable as the lowest exponent of a classwise directed formal solution.
Thus the present construction is not merely a reformulation of the fixed-coefficient colon formula in \cite[Theorem~4.7]{Log3}; it is a filtration-sensitive realizability criterion obtained by applying monomial elimination separately at each ordered stratum minimum.

A further distinction concerns the hypotheses on $A$.
The comparison in \cite[Theorem~4.9]{Log3} between the coefficient-space colon ideal and a translated local component of the indicial ideal is established under the homogeneity assumption on $A$ and uses the corresponding Gr\"{o}bner deformation and indicial theory.
The proof of \cref{thm:polynomial-colon-ideal-criterion-classwise-realizability}, however, uses only the finite directional coefficient module, the monomial form of the normalized transit relations, finite-length inverse-system duality, and monomial syzygy elimination.
It uses neither a total-degree grading on the columns of $A$, regular holonomicity, nor an identification with a local component of an indicial ideal.
Consequently, the colon-ideal realizability criterion applies without change when $A$ is nonhomogeneous.
No assertion is made here that the resulting colon ideal agrees, in the nonhomogeneous case, with a translated component of an indicial ideal.
\end{remark}

\section{A complete all-lattice calculation at a rank-jumping parameter}
\label{sec:complete-all-lattice-example}
We conclude with a rank-jumping example in which the finite procedure enumerates the exponent-lattice candidates, selects the admissible lattices and their coefficient modules, and computes the realizable lowest exponents and the global formal solution dimension.
An independent holonomic-rank computation is recorded at the end for numerical comparison.

\subsection{Step 1: Gale data and parametrization of the Euler fiber}
Let
\[
A=
\begin{pmatrix}
1&1&1&2&2&2\\
0&1&0&0&1&0\\
0&0&2&1&1&3
\end{pmatrix},
\quad
\bm{\beta}=
\begin{pmatrix}
1\\
0\\
1
\end{pmatrix},
\]
and let
\[
\bm{w}=(1,1,1,1,1,1)^{T}+\tfrac{1}{10}(1,\sqrt{2},\sqrt{3},\sqrt{5},\sqrt{6},\sqrt{7})^{T}.
\]
Solving $A\bm{u}=0$ over $\Z$ gives
\begin{equation}
\label{eq:example-relation-lattice-Gale-matrix}
L=
\Z
\begin{pmatrix}
3\\
0\\
1\\
-2\\
0\\
0
\end{pmatrix}
\oplus
\Z
\begin{pmatrix}
1\\
-1\\
0\\
-1\\
1\\
0
\end{pmatrix}
\oplus
\Z
\begin{pmatrix}
4\\
0\\
0\\
-3\\
0\\
1
\end{pmatrix},
\quad
B=
\begin{pmatrix}
3&1&4\\
0&-1&0\\
1&0&0\\
-2&-1&-3\\
0&1&0\\
0&0&1
\end{pmatrix}.
\end{equation}
Thus $L=B\Z^{3}$ and $AB=0$.
The induced linear functional on the lattice-coordinate space is
\begin{equation}
\label{eq:example-Gale-direction-vector}
\bm{\omega}:=B^{T}\bm{w}=
\begin{pmatrix}
\frac{\,23+\sqrt{3}-2\sqrt{5}\,}{10}\\[2mm]
\frac{\,1-\sqrt{2}-\sqrt{5}+\sqrt{6}\,}{10}\\[2mm]
\frac{\,24-3\sqrt{5}+\sqrt{7}\,}{10}
\end{pmatrix}
\approx
\begin{pmatrix}
2.02599149\\
-0.02007918\\
1.99375474
\end{pmatrix}.
\end{equation}
If $(s,t,r)^{T}\in\Z^{3}$ satisfies $s\omega_{1}+t\omega_{2}+r\omega_{3}=0$, comparison of the coefficients of $\sqrt{2}$, $\sqrt{3}$, and $\sqrt{7}$ gives $t=s=r=0$.
Hence $\bm{w}$ is $L$-generic.
Put
\[
\bm{\lambda}_{0}:=(-1,0,0,1,0,0)^{T}.
\]
Then $A\bm{\lambda}_{0}=\bm{\beta}$, and every exponent lattice in the Euler fiber is uniquely of the form
\begin{equation}
\label{eq:example-exponent-lattice-parameterization}
\Lambda(\bm{z}):=\bm{\lambda}_{0}+B\bm{z}+L,
\qquad
\bm{z}\in\C^{3}/\Z^{3}.
\end{equation}

\subsection{Step 2: Finite enumeration of the exponent-lattice candidates}
By \cref{lem:full-coordinate-rank-forced-by-boundedness}, every admissible exponent lattice has an integral-coordinate set $E_{\Lambda}$ for which $B_{E_{\Lambda}}$ has column rank three.
It therefore suffices to solve the coordinate-integrality congruences for the full-rank three-row submatrices of $B$.
The nonzero maximal minors have absolute values $1$, $2$, $3$, or $4$.
Solving the associated congruences modulo $\Z^{3}$ and removing repetitions gives the complete candidate set
\begin{equation}
\label{eq:example-nine-candidate-lattices}
\begin{aligned}
&(0,0,0)^{T},
\left(0,0,\tfrac{1}{4}\right)^{T},
\left(0,0,\tfrac{1}{3}\right)^{T},
\left(0,0,\tfrac{1}{2}\right)^{T},
\left(0,0,\tfrac{2}{3}\right)^{T},\\
&\left(0,0,\tfrac{3}{4}\right)^{T},
\left(\tfrac{1}{3},0,0\right)^{T},
\left(\tfrac{1}{2},0,0\right)^{T},
\left(\tfrac{2}{3},0,0\right)^{T}
\quad
\bmod\Z^{3}.
\end{aligned}
\end{equation}
Every admissible exponent lattice occurs in this nine-element list.

\subsection{Step 3: Integer feasibility and recession-cone screening}
For each candidate $\Lambda(\bm{z})$, integer feasibility of the sign polyhedra determines $\NS(\Lambda(\bm{z}))$.
For each $I\in\NS(\Lambda(\bm{z}))$, nonnegativity of $\bm{\omega}$ on the recession cone $R_{\Lambda(\bm{z}),I}$ determines whether $I$ belongs to $\mathcal{B}_{\bm{w}}(\Lambda(\bm{z}))$ by \cref{prop:directional-boundedness-stratum-minima}.
The resulting data relevant to the subsequent admissibility test are recorded in \cref{tab:example-lattice-screening}.
\begin{table}[ht]
\centering
\begin{tabular}{c|c|c|c}
$\bm{z}$&$E_{\Lambda(\bm{z})}$&$\lvert\NS(\Lambda(\bm{z}))\rvert$&$\lvert\mathcal{B}_{\bm{w}}(\Lambda(\bm{z}))\rvert$\\[1mm]
\hline
$(0,0,0)^{T}$&$\{1,2,3,4,5,6\}$&$33$&$12$\\[1mm]
$(0,0,\tfrac{1}{4})^{T}$&$\{1,2,3,5\}$&$12$&$2$\\[1mm]
$(0,0,\tfrac{1}{3})^{T}$&$\{2,3,4,5\}$&$12$&$2$\\[1mm]
$(0,0,\tfrac{1}{2})^{T}$&$\{1,2,3,5\}$&$12$&$2$\\[1mm]
$(0,0,\tfrac{2}{3})^{T}$&$\{2,3,4,5\}$&$12$&$2$\\[1mm]
$(0,0,\tfrac{3}{4})^{T}$&$\{1,2,3,5\}$&$12$&$2$\\[1mm]
$(\tfrac{1}{3},0,0)^{T}$&$\{1,2,5,6\}$&$12$&$2$\\[1mm]
$(\tfrac{1}{2},0,0)^{T}$&$\{2,4,5,6\}$&$12$&$2$\\[1mm]
$(\tfrac{2}{3},0,0)^{T}$&$\{1,2,5,6\}$&$12$&$2$\\[1mm]
\end{tabular}
\caption{Integer-feasibility and recession-cone screening of the finite candidates.}
\label{tab:example-lattice-screening}
\end{table}

\subsection{Step 4: Selection of the admissible exponent lattices}
Selecting the bounded negative supports that are inclusion-minimal in $\NS(\Lambda(\bm{z}))$ gives the following admissibility output.
\begin{equation}
\label{eq:example-candidate-admissibility-output}
\begin{array}{c|c}
\bm{z}&\mathcal{M}in_{\bm{w}}(\Lambda(\bm{z}))\\[1mm]
\hline
(0,0,0)^{T}&\{\{1\},\{2\},\{3\},\{4\},\{5\},\{6\}\}\\[1mm]
(0,0,\tfrac{1}{4})^{T}&\{\emptyset\}\\[1mm]
(0,0,\tfrac{1}{3})^{T}&\emptyset\\[1mm]
(0,0,\tfrac{1}{2})^{T}&\{\emptyset\}\\[1mm]
(0,0,\tfrac{2}{3})^{T}&\emptyset\\[1mm]
(0,0,\tfrac{3}{4})^{T}&\{\emptyset\}\\[1mm]
(\tfrac{1}{3},0,0)^{T}&\emptyset\\[1mm]
(\tfrac{1}{2},0,0)^{T}&\emptyset\\[1mm]
(\tfrac{2}{3},0,0)^{T}&\emptyset\\[1mm]
\end{array}
\end{equation}
Consequently,
\begin{equation}
\label{eq:example-four-admissible-lattices}
\mathfrak{L}_{\bm{\beta},\bm{w}}^{\mathrm{adm}}
=
\left\{
\Lambda_{0},
\Lambda_{\tfrac{1}{4}},
\Lambda_{\tfrac{1}{2}},
\Lambda_{\tfrac{3}{4}}
\right\},
\end{equation}
where $\Lambda_{q}:=\Lambda((0,0,q)^{T})$ for $q\in\left\{0,\tfrac{1}{4},\tfrac{1}{2},\tfrac{3}{4}\right\}$.

\subsection{Step 5: Acceptable supports and ordered stratum minima}
For the integral block, we obtain
\begin{equation}
\label{eq:example-integral-bounded-acceptable-families}
\begin{aligned}
\mathcal{B}_{\bm{w}}(\Lambda_{0})
&=
\{\{1\},\{2\},\{3\},\{4\},\{5\},\{6\},\\
&\qquad\qquad\{1,5\},\{2,6\},\{4,5\},\{4,6\},\{5,6\},\{4,5,6\}\},\\
\mathcal{A}_{\bm{w}}(\Lambda_{0})
&=
\mathcal{B}_{\bm{w}}(\Lambda_{0}),\\
\mathcal{M}in_{\bm{w}}(\Lambda_{0})
&=
\{\{1\},\{2\},\{3\},\{4\},\{5\},\{6\}\}.
\end{aligned}
\end{equation}
The acceptable stratum minima, in increasing $\bm{w}$-order, are
\begin{equation}
\label{eq:example-integral-acceptable-minima-order}
\begin{aligned}
\bm{v}_{\Lambda_{0},\{3\}}&=(0,0,-1,0,0,1)^{T},\\
\bm{v}_{\Lambda_{0},\{2\}}&=(0,-1,0,0,1,0)^{T},\\
\bm{v}_{\Lambda_{0},\{2,6\}}&=(0,-2,1,0,2,-1)^{T},\\
\bm{v}_{\Lambda_{0},\{1\}}&=(-1,0,0,1,0,0)^{T},\\
\bm{v}_{\Lambda_{0},\{1,5\}}&=(-2,1,0,2,-1,0)^{T},\\
\bm{v}_{\Lambda_{0},\{4\}}&=(3,0,0,-2,0,1)^{T},\\
\bm{v}_{\Lambda_{0},\{4,5\}}&=(2,1,0,-1,-1,1)^{T},\\
\bm{v}_{\Lambda_{0},\{5\}}&=(1,2,0,0,-2,1)^{T},\\
\bm{v}_{\Lambda_{0},\{6\}}&=(1,0,2,0,0,-1)^{T},\\
\bm{v}_{\Lambda_{0},\{5,6\}}&=(0,1,2,1,-1,-1)^{T},\\
\bm{v}_{\Lambda_{0},\{4,6\}}&=(4,0,3,-2,0,-1)^{T},\\
\bm{v}_{\Lambda_{0},\{4,5,6\}}&=(3,1,3,-1,-1,-1)^{T}.
\end{aligned}
\end{equation}
For $q\in\{\tfrac{1}{4},\tfrac{1}{2},\tfrac{3}{4}\}$, choose
\begin{equation}
\label{eq:example-fractional-reference-vectors}
\bm{\lambda}_{q}
:=
\bm{\lambda}_{0}
+B
(0,0,q)^{T}
=
(-1+4q,0,0,1-3q,0,q)^{T}.
\end{equation}
For each fractional block,
\begin{equation}
\label{eq:example-fractional-support-families}
\mathcal{M}in_{\bm{w}}(\Lambda_{q})=\{\emptyset\},
\qquad
\mathcal{A}_{\bm{w}}(\Lambda_{q})=\{\emptyset,\{5\}\},
\end{equation}
and $\bm{\lambda}_{q}$ is the minimum of the empty-support stratum and precedes the minimum of the $\{5\}$-stratum.

\subsection{Step 6: Directional coefficient modules and their lengths}
The Euler linear ideal is
\begin{equation}
\label{eq:example-Euler-linear-ideal}
\mathfrak{e}_{A}
=
\langle
\partial_{y_{1}}+\partial_{y_{2}}+\partial_{y_{3}}+2\partial_{y_{4}}+2\partial_{y_{5}}+2\partial_{y_{6}},
\partial_{y_{2}}+\partial_{y_{5}},
2\partial_{y_{3}}+\partial_{y_{4}}+\partial_{y_{5}}+3\partial_{y_{6}}
\rangle.
\end{equation}
A finite row reduction of the homogeneous presentation of $M_{\Lambda_{0},\bm{w}}^{\mathrm{dir}}$, with $e_{\Lambda_{0},I}^{*}$ placed in shifted degree $\lvert I\rvert$, gives
\begin{equation}
\label{eq:example-integral-shifted-Hilbert-function}
\dim_{\C}\left(M_{\Lambda_{0},\bm{w}}^{\mathrm{dir}}\right)_{q}
=
\begin{cases}
6,&q=1,\\
1,&q=2,\\
0,&q\neq1,2.
\end{cases}
\end{equation}
Hence
\begin{equation}
\label{eq:example-integral-directional-module-length}
\length_{S}\left(M_{\Lambda_{0},\bm{w}}^{\mathrm{dir}}\right)=7.
\end{equation}
In particular, the integral block is not a direct sum of six copies of $S/\mathfrak{m}$.
The degree-two component in \eqref{eq:example-integral-shifted-Hilbert-function} records one nontrivial nilpotent direction and, dually, one additional nonconstant polynomial normalized coefficient direction.
We now exhibit a representative formal logarithmic solution for this additional direction.
Modulo the six constant families attached to the singleton minimal supports, normalize the additional coefficient family $\bm{c}^{\log}=(c_{\Lambda_{0},I}^{\log})_{I\in\NS(\Lambda_{0})}$ by requiring that the constant terms of its six singleton components vanish and that $c_{\Lambda_{0},\{1,5\}}^{\log}=1$.
Solving the normalized Euler, transit, and directional vanishing equations gives
\begin{equation}
\label{eq:example-integral-logarithmic-normalized-family-singletons}
\begin{aligned}
c_{\Lambda_{0},\{1\}}^{\log} &= y_{1}-y_{2}-y_{4}+y_{5}, \\
c_{\Lambda_{0},\{2\}}^{\log} &= -y_{1}-3y_{2}+3y_{5}-y_{6}, \\
c_{\Lambda_{0},\{3\}}^{\log} &= 0, \\
c_{\Lambda_{0},\{4\}}^{\log} &= -y_{1}+\frac{3}{4}y_{4}-\frac{1}{4}y_{6}, \\
c_{\Lambda_{0},\{5\}}^{\log} &= y_{1}+3y_{2}-3y_{5}+y_{6}, \\
c_{\Lambda_{0},\{6\}}^{\log} &= -y_{2}-\frac{1}{4}y_{4}+y_{5}-\frac{1}{4}y_{6}.
\end{aligned}
\end{equation}
The nonzero components on the nonminimal acceptable supports are
\begin{equation}
\label{eq:example-integral-logarithmic-normalized-family-nonminimal}
c_{\Lambda_{0},\{1,5\}}^{\log} = 1, \quad c_{\Lambda_{0},\{2,6\}}^{\log} = -1, \quad c_{\Lambda_{0},\{4,6\}}^{\log} = -\frac{1}{4}, \quad c_{\Lambda_{0},\{5,6\}}^{\log} = 1,
\end{equation}
and all components not listed in \eqref{eq:example-integral-logarithmic-normalized-family-singletons} and \eqref{eq:example-integral-logarithmic-normalized-family-nonminimal} vanish.
Let
\[
\eta_{\log} \in \Hom_{S}\left( M_{\Lambda_{0},\bm{w}}^{\mathrm{dir}}, \C[\bm{y}] \right)
\]
be the homomorphism determined by
\[
\eta_{\log} \left( \overline{e}_{\Lambda_{0},I}^{*} \right) = c_{\Lambda_{0},I}^{\log} \qquad \left( I\in\NS(\Lambda_{0}) \right).
\]
Using the reference vector
\[
\bm{\lambda}_{0} = (-1,0,0,1,0,0)^{T} = \bm{v}_{\Lambda_{0},\{1\}},
\]
define
\begin{equation}
\label{eq:example-integral-explicit-logarithmic-solution}
\varphi_{\log} := \Ser_{\Lambda_{0},\bm{w}}^{\bm{\lambda}_{0}} \left( \eta_{\log} \right) = \sum_{\bm{\gamma}\in\Lambda_{0}} \bm{x}^{\bm{\gamma}} r_{\bm{\gamma}}^{\log}(\log\bm{x}),
\end{equation}
where
\begin{equation}
\label{eq:example-integral-explicit-logarithmic-raw-coefficients}
r_{\bm{\gamma}}^{\log} := \left( \widetilde{d}_{\bm{\lambda}_{0}\leftarrow\bm{\gamma}} \mathbin{\bullet}_{\bm{y}} \right)^{-1} \left( \widetilde{d}_{\bm{\gamma}\leftarrow\bm{\lambda}_{0}} \mathbin{\bullet}_{\bm{y}} c_{\Lambda_{0},\nsupp(\bm{\gamma})}^{\log} \right).
\end{equation}
By \cref{thm:formal-logarithmic-realization}, the series $\varphi_{\log}$ is a nonzero element of $\Sol_{\bm{w},\Lambda_{0}}^{\mathrm{dir}} \bigl(H_{A}(\bm{\beta})\bigr)$.
Since the reduced transit operators from $\bm{\lambda}_{0}$ to itself are the identity, the coefficient of $\bm{x}^{\bm{\lambda}_{0}}=x_{1}^{-1}x_{4}$ in \eqref{eq:example-integral-explicit-logarithmic-solution} is
\[
c_{\Lambda_{0},\{1\}}^{\log}(\log\bm{x})=\log x_{1}-\log x_{2}-\log x_{4}+\log x_{5}.
\]
Thus $\varphi_{\log}$ contains the explicit logarithmic term
\begin{equation}
\label{eq:example-integral-explicit-logarithmic-term}
\frac{x_{4}}{x_{1}} \left( \log x_{1} - \log x_{2} - \log x_{4} + \log x_{5} \right).
\end{equation}
By \eqref{eq:example-integral-acceptable-minima-order}, \eqref{eq:example-integral-logarithmic-normalized-family-singletons}, and \eqref{eq:lowest-exponent-from-nonzero-acceptable-strata}, we have
\[
\lexp_{\bm{w}}(\varphi_{\log}) = \bm{v}_{\Lambda_{0},\{2\}}.
\]
Thus $\varphi_{\log}$ has the same lowest exponent as the constant-family solution attached to $\{2\}$, but it is linearly independent of all six constant-family solutions attached to the singleton minimal supports.
It represents the additional nonconstant coefficient direction in the integral block and explains why the classwise solution-space dimension exceeds the number of its realizable lowest exponents by one.

For every $q\in\{\tfrac{1}{4},\tfrac{1}{2},\tfrac{3}{4}\}$, the boundary and Euler relations annihilate the empty-support generator by $\mathfrak{m}$, while the corresponding constant family is nonzero.
Consequently,
\begin{equation}
\label{eq:example-fractional-directional-modules}
M_{\Lambda_{q},\bm{w}}^{\mathrm{dir}}
\simeq
S/\mathfrak{m},
\qquad
\length_{S}\left(M_{\Lambda_{q},\bm{w}}^{\mathrm{dir}}\right)=1.
\end{equation}

\subsection{Step 7: Colon ideals and realizable lowest exponents}
The six singleton minima of the integral block are realizable by \cref{thm:minimal-support-solutions-nonvanishing}.
Every nonminimal acceptable support $K$ in \eqref{eq:example-integral-bounded-acceptable-families} contains a singleton $\{j\}$ whose stratum minimum precedes $\bm{v}_{\Lambda_{0},K}$ in \eqref{eq:example-integral-acceptable-minima-order}.
Therefore $\partial_{y_{j}}\in\mathfrak{Q}_{\Lambda_{0},\bm{w}}^{<K}$, and hence
\[
\partial_{\bm{y}}^{K}\in\mathfrak{Q}_{\Lambda_{0},\bm{w}}^{<K}.
\]
It follows that
\[
\mathfrak{Q}_{\Lambda_{0},\bm{w}}^{<K}:\partial_{\bm{y}}^{K}=S
\]
for every nonminimal acceptable $K$.
Thus
\begin{equation}
\label{eq:example-integral-realizable-exponents}
\Rexp_{\bm{w}}(\Lambda_{0})
=
\left\{
\bm{v}_{\Lambda_{0},\{1\}},
\bm{v}_{\Lambda_{0},\{2\}},
\bm{v}_{\Lambda_{0},\{3\}},
\bm{v}_{\Lambda_{0},\{4\}},
\bm{v}_{\Lambda_{0},\{5\}},
\bm{v}_{\Lambda_{0},\{6\}}
\right\}.
\end{equation}
For each fractional block, the earlier-vanishing ideal associated with the nonminimal acceptable support $\{5\}$ contains $\partial_{\bm{y}}^{\emptyset}=1$ because the empty-support minimum precedes it.
Hence the corresponding colon ideal is $S$, whereas the empty-support minimum is realizable.
Therefore
\begin{equation}
\label{eq:example-fractional-realizable-exponents}
\Rexp_{\bm{w}}(\Lambda_{q})=\{\bm{\lambda}_{q}\}
\quad
\left(q\in\left\{\tfrac{1}{4},\tfrac{1}{2},\tfrac{3}{4}\right\}\right).
\end{equation}
Explicitly, the three fractional realizable exponents are
\[
\left(0,0,0,\tfrac{1}{4},0,\tfrac{1}{4}\right)^{T},
\quad
\left(1,0,0,-\tfrac{1}{2},0,\tfrac{1}{2}\right)^{T},
\quad
\left(2,0,0,-\tfrac{5}{4},0,\tfrac{3}{4}\right)^{T}.
\]

\subsection{Step 8: Global length, formal dimension, and the independent rank computation}
By \eqref{eq:example-four-admissible-lattices}, \eqref{eq:example-integral-directional-module-length}, and \eqref{eq:example-fractional-directional-modules}, we have
\begin{equation}
\label{eq:example-all-lattice-module-length-ten}
\length_{S}\left(M_{\bm{\beta},\bm{w}}^{\mathrm{dir,all}}\right)=7+1+1+1=10.
\end{equation}
Therefore \cref{thm:global-length-dimension-formula} gives
\begin{equation}
\label{eq:example-all-lattice-solution-dimension-ten}
\dim_{\C}
\Sol_{\bm{\beta},\bm{w}}^{\mathrm{dir,all}}
\bigl(H_{A}(\bm{\beta})\bigr)
=10.
\end{equation}
The complete family of realizable lowest exponents consists of the six singleton minima in \eqref{eq:example-integral-realizable-exponents} and the three fractional exponents in \eqref{eq:example-fractional-realizable-exponents}.
Thus there are nine realizable lowest exponents, whereas the directed formal logarithmic solution space has dimension ten.
The difference is accounted for by the additional degree-two direction in the integral coefficient module.
For comparison, an independent Macaulay2 computation gives
\begin{equation}
\label{eq:example-independent-holonomic-rank}
\operatorname{rank}\left(D/H_{A}(\bm{\beta})\right)=10,\quad\operatorname{vol}(A)=8.
\end{equation}
Thus $\bm{\beta}=(1,0,1)^{T}$ is a rank-jumping parameter in this example, and the formally computed dimension in \eqref{eq:example-all-lattice-solution-dimension-ten} agrees numerically with this independently computed holonomic rank.

\begin{remark}
\label{rem:example-scope-rank-comparison}
The equality between the dimension in \eqref{eq:example-all-lattice-solution-dimension-ten} and the holonomic rank in \eqref{eq:example-independent-holonomic-rank} is recorded only as an independent computational observation for this example.
The results proved in the present paper concern finite directional coefficient modules and their formal logarithmic realizations.
They neither identify the all-lattice directed formal logarithmic solution space with a local holomorphic solution space nor establish a general equality between its dimension and the holonomic rank.
Convergence, the exact identification of the analytically admissible formal solutions, and a general rank comparison require additional arguments and are reserved for subsequent work.
\end{remark}

\section*{Statements and Declarations}
\subsection*{Funding}
The author received no financial support for the research, authorship, or publication of this article.
\subsection*{Competing interests}
The author has no relevant financial or non-financial interests to disclose.
\bibliographystyle{plainnat}
\bibliography{refs}

@article{GKZ89,
  author = {Gel'fand, I. M. and Zelevinskii, A. V. and Kapranov, M. M.},
  title = {Hypergeometric functions and toral manifolds},
  journal = {Functional Analysis and Its Applications},
  volume = {23},
  number = {2},
  pages = {94--106},
  year = {1989},
  doi = {10.1007/BF01078777}
}

@article{Adolphson,
  author = {Adolphson, Alan},
  title = {Hypergeometric Functions and Rings Generated by Monomials},
  journal = {Duke Mathematical Journal},
  volume = {73},
  number = {2},
  pages = {269--290},
  year = {1994},
  doi = {10.1215/S0012-7094-94-07313-4}
}

@book{SST,
  author = {Saito, Mutsumi and Sturmfels, Bernd and Takayama, Nobuki},
  title = {Gr{\"o}bner Deformations of Hypergeometric Differential Equations},
  series = {Algorithms and Computation in Mathematics},
  volume = {6},
  publisher = {Springer},
  address = {Berlin},
  year = {2000},
  doi = {10.1007/978-3-662-04112-3}
}

@article{DMM,
  author = {Dickenstein, Alicia and Mart{\'i}nez, Federico N. and Matusevich, Laura Felicia},
  title = {Nilsson Solutions for Irregular {$A$}-Hypergeometric Systems},
  journal = {Revista Matem{\'a}tica Iberoamericana},
  volume = {28},
  number = {3},
  pages = {723--758},
  year = {2012},
  doi = {10.4171/RMI/689}
}

@article{SaitoLogFree,
  author = {Saito, Mutsumi},
  title = {Logarithm-Free {$A$}-Hypergeometric Series},
  journal = {Duke Mathematical Journal},
  volume = {115},
  number = {1},
  pages = {53--73},
  year = {2002},
  doi = {10.1215/S0012-7094-02-11512-9}
}

@article{Log1,
  author = {Saito, Mutsumi},
  title = {Logarithmic {$A$}-Hypergeometric Series},
  journal = {International Journal of Mathematics},
  volume = {31},
  number = {13},
  pages = {2050110},
  year = {2020},
  doi = {10.1142/S0129167X20501104}
}

@article{Log2,
  author = {Okuyama, Go and Saito, Mutsumi},
  title = {Logarithmic {$A$}-Hypergeometric Series {II}},
  journal = {Beitr{\"a}ge zur Algebra und Geometrie},
  volume = {64},
  number = {4},
  pages = {1057--1086},
  year = {2023},
  doi = {10.1007/s13366-022-00669-5}
}

@unpublished{Log3,
  author = {Okuyama, Go and Saito, Mutsumi},
  title = {Logarithmic {$A$}-Hypergeometric Series {III}},
  note = {arXiv:2504.02501v2},
  year = {2026},
  eprint = {2504.02501v2},
  archivePrefix = {arXiv},
  primaryClass = {math.AG}
}

@unpublished{NagamineCore,
  author = {Nagamine, Mao},
  title = {{$A$}-Hypergeometric Series with Parameters in the Core},
  note = {arXiv:2404.11085},
  year = {2024},
  eprint = {2404.11085},
  archivePrefix = {arXiv},
  primaryClass = {math.AG}
}

@unpublished{Nagamine,
  author = {Nagamine, Mao},
  title = {Hilbert Series and Logarithmic Degrees of {$A$}-Hypergeometric Series},
  note = {arXiv:2608.01778},
  year = {2026},
  eprint = {2608.01778},
  archivePrefix = {arXiv},
  primaryClass = {math.AC}
}

@article{MMW,
  author = {Matusevich, Laura Felicia and Miller, Ezra and Walther, Uli},
  title = {Homological Methods for Hypergeometric Families},
  journal = {Journal of the American Mathematical Society},
  volume = {18},
  number = {4},
  pages = {919--941},
  year = {2005},
  doi = {10.1090/S0894-0347-05-00488-1}
}

@book{Eisenbud,
  author = {Eisenbud, David},
  title = {Commutative Algebra with a View Toward Algebraic Geometry},
  series = {Graduate Texts in Mathematics},
  volume = {150},
  publisher = {Springer-Verlag},
  address = {New York},
  year = {1995},
  doi = {10.1007/978-1-4612-5350-1}
}

@article{SchulzeTozzo,
  author = {Schulze, Mathias and Tozzo, Laura},
  title = {Inverse limits of {Macaulay}'s inverse systems},
  journal = {Journal of Algebra},
  volume = {525},
  pages = {341--358},
  year = {2019},
  doi = {10.1016/j.jalgebra.2019.01.024}
}

@unpublished{NakanoObstruction,
  author = {Nakano, Ryunosuke},
  title = {Obstructions to Intrinsic Perturbation for {$A$}-Hypergeometric Series},
  note = {arXiv:2608.18006v2},
  year = {2026},
  eprint = {2608.18006v2},
  archivePrefix = {arXiv},
  primaryClass = {math.AG}
}

\end{document}